\documentclass[unicode, 12pt, a4paper]{article}

\usepackage{amsmath,amsthm}
\usepackage[colorlinks=true, urlcolor=blue, pdfborder={0 0 0}]{hyperref}
\usepackage[T2A]{fontenc}
\usepackage[utf8]{inputenc}
\usepackage[english,russian]{babel}
\usepackage{mathtools}
\usepackage[framemethod=1]{mdframed}
\usepackage{lipsum}

\usepackage{dsfont}

\usepackage{tikz}
\usepackage{amsfonts}

\usepackage{makecell}

\usepackage[left=3cm,right=3cm,
    top=3cm,bottom=3cm,bindingoffset=0cm]{geometry}

\allowdisplaybreaks[1]

\newcommand*\circled[1]{\tikz[baseline=(char.base)]{
		\node[shape=circle,draw,inner sep=2pt] (char) {#1};}}

\newcommand{\norm}[1]{\left\lVert#1\right\rVert}
\newcommand\abs[1]{\left|#1\right|}
\newcommand\uprule{\rule{0mm}{1.9ex}}
\newcommand{\argmin}{\operatornamewithlimits{argmin}}

\newtheorem{theorem}{Теорема}
\newtheorem{lemma}{Лемма}
\newtheorem{corollary}{Следствие}
\newtheorem{definition}{Определение}
\newtheorem{remark}{Замечание}

\DeclarePairedDelimiter\ceil{\lceil}{\rceil}

\newcommand\tline[4]{$\underset{\text{#1}}{\text{\underline{\hspace{#2} #4 \hspace{#3}}}}$}

\begin{document}

\begin{titlepage}
	\begin{center}
		\bfseries
		ПРАВИТЕЛЬСТВО РОССИЙСКОЙ ФЕДЕРАЦИИ\\
		ФЕДЕРАЛЬНОЕ ГОСУДАРСТВЕННОЕ АВТОНОМНОЕ\\
		ОБРАЗОВАТЕЛЬНОЕ УЧРЕЖДЕНИЕ ВЫСШЕГО ОБРАЗОВАНИЯ\\
		НАЦИОНАЛЬНЫЙ ИССЛЕДОВАТЕЛЬСКИЙ УНИВЕРСИТЕТ\\
		«ВЫСШАЯ ШКОЛА ЭКОНОМИКИ»
	\end{center}
	
	\begin{center}
		Факультет компьютерных наук 
	\end{center}
	
	\hfill
	\begin{minipage}{.45\linewidth}
	\begin{center}
	УТВЕРЖДАЮ\\
	Академический руководитель\\
	образовательной программы\\
	«Математические методы оптимизации и стохастики»,\\~\\
	\underline{\hspace{3cm}} В.Г.Спокойный\\
	«\underline{\hspace{0.5cm}}» \underline{\hspace{3cm}} 2017г.
	\end{center}
	\end{minipage}
	
	\begin{center}
	\bfseries
		Выпускная квалификационная работа 
	\end{center}
	на тему \\
	«Зеркальный вариант метода подобных треугольников для задач условной оптимизации»\\~\\
	тема на английском языке \\
	«Mirror version of similar triangles method for constrained optimization problems»
	
	\begin{center}
		по направлению подготовки 01.04.02 «Математические методы оптимизации и стохастики» 
	\end{center}
	
	\begin{center}
	\begin{tabular}[t]{|l|l|}
	\hline
	\makecell[l]{
	Научный руководитель\\~\\
	\tline{Должность, место работы}{0.25in}{0.25in}{Доцент, НИУ ВШЭ}\\
	\tline{ученая степень, ученое звание}{0.66in}{0.67in}{Д.ф.-м.н}\\
	\tline{И.О. Фамилия}{0.46in}{0.46in}{А.В. Гасников}\\
	\tline{Оценка}{1in}{0.97in}{}\\
	\tline{Подпись, Дата}{1in}{0.97in}{}\\
	}
	&
		\makecell[l]{
		Выполнил\\
		студент группы М15МОС\\
		2 курса магистратуры\\
		образовательной программы\\ «Математические методы\\ оптимизации и стохастики»\\~\\
		\tline{И.О. Фамилия}{0.53in}{0.54in}{A.И. Тюрин}\\
		\tline{Подпись, Дата}{1in}{0.97in}{}\\
		} \\
	\hline
	\end{tabular}
	\end{center}
	
	\begin{center}
		\bf{Москва 2017}
	\end{center}

\end{titlepage}

\begin{center}Аннотация\end{center}

Наука о методах оптимизации является бурно развивающей в наше время. В машинном обучении, компьютерном зрении, биологии, медицине, конструировании и во многих других отраслях методы оптимизации имеют огромную популярность и являются одним из важнейших инструментов. Одна из основных целей науки: получить некоторый "универсальный"\, метод, который будет хорошо работать на всех задачах, независимо от гладкости задачи, точности вычисления градиента и других параметров, который характеризуют задачу. В нашей работе мы предлагаем метод, представляющий из себя "универсальный"\, для многих постановок, но при этом он прост в изложении и понимании.
\pagebreak

\begin{center}Abstract\end{center}
Science about optimization methods is rapidly developing today. In machine learning, computer vision, biology, medicine, construction and in many other different areas optimization methods have vast popularity and they appear as important tool. One of the most important goals in optimization: create some "universal"\, method, which will have good performance in all problems regardless smoothness of a task, computation precision of gradient and other parameters which characterize a problem. In this thesis we propose a method which is "universal"\, for different problems and, at the same time, is simple for understanding.
\pagebreak

\tableofcontents
\newpage

\section{Введение}

В данной работе подробно описан метод, называемый зеркальным методом треугольника, этот метод был получен по аналогии с оригинальным методом треугольника \cite{gasnikov2016universal}. Основное отличие заключается в том, что в методе из \cite{gasnikov2016universal} во вспомогательном шаге происходит накопление градиентов, как это делается, например, в \cite{nesterov2009primal}. В нашей же работе вспомогательный шаг имеет структуру зеркального спуска \cite{beck2003mirror} \cite{duchi2010composite}. Зеркальный метод треугольника (ЗМТ) был успешно обобщен на различные оптимизационные задачи, об этих обобщениях рассказывается в данной работе. Структура выпускной квалификационной работы следующая: в разделе \ref{sec:mmt} описан базовый метод для общей задачи оптимизации, в разделе \ref{sec:minmax} описано обобщение ЗМТ для задачи минимакса с адаптивностью, в разделе \ref{sec:mmtDL} рассказывается о ЗМТ с $(\delta, L)$-оракулом \cite{devolder2013exactness}, в разделе \ref{sec:rand} о решении оптимизационной задачи с оракулом, который выдает смещенную рандомизированную оценку истинного градиента.

\section{Зеркальный метод треугольника} \label{sec:mmt}

Введем для начала общую постановку задачи гладкой выпуклой оптимизации \cite{nesterov2010methods}. Пусть $\mathds{R}^n$ - евклидово конечномерное действительное векторное пространство с произвольной нормой $\norm{}$. Пусть определена функция $f(x): Q \longrightarrow \mathds{R}$.  

Будем полагать, что
\begin{enumerate}
	\item $Q \subseteq \mathds{R}^n$, выпуклое, замкнутое.
	\item $f(x)$ - непрерывная и выпуклая функция на $Q$.
	\item $f(x)$ ограничена снизу на $Q$ и достигает своего минимума $f_*$ в некоторой точке (необязательно одной) $x_* \in Q$.
	\item $\nabla f(x)$ существует на $Q$ и является липшицевым с константой $L$, то есть
	\begin{align}
	\norm{\nabla f(x) - \nabla f(y)}_* \leq L\norm{x - y} \,\,\,\,\,\, \forall x, y \in Q.
	\end{align}
	Где $\norm{\lambda}_* = \max\limits_{\norm{\nu} \leq 1;\nu \in \mathds{R}^n}\langle \lambda,\nu\rangle\,\,\,\forall \lambda \in \mathds{R}^n$.
\end{enumerate}

Рассматривается следующая задача оптимизации:
\begin{align}
\label{mainTask}
f(x) \rightarrow \min_{x \in Q}.
\end{align}

Введем два понятия: прокс-функция и расстояние Брегмана \cite{gupta2008bregman}.

\begin{definition}
$d(x):Q \rightarrow \mathds{R}$ называется прокс-функцией, если $d(x)$ непрерывно дифференцируемая на $Q$ и $d(x)$ является 1 - сильно выпуклой относительно нормы $\norm{}$.
\end{definition}

В общем случае определение прокс-функции немного сложнее, но для класса задач, которые мы решаем, этого достаточно.

\begin{definition}
Расстоянием Брегмана называется 
\begin{align}
V(x,y) = d(x) - d(y) - \langle\nabla d(y), x - y\rangle,
\end{align}
где $d(x)$ - произвольная прокс-функция.
\end{definition}

Легко показать, что $V(x,y) \geq \frac{1}{2}\norm{x - y}^2.$

Во всех далее предложенных алгоритмах имеется некоторая точка, с которой начинается работа метода.
\begin{definition}
$x_0$ - начальная точка работы алгоритма.
\end{definition}

\begin{definition}
Обозначим за $R^2$ такое число, что  $V(x_*, x_0) \leq R^2$.
\end{definition}

 Рассмотрим алгоритм зеркального метода треугольника.

\begin{mdframed}
\textbf{Дано:} $x_0$ - начальная точка, $N$ - количество шагов метода и $L$ - константа Липшица $\nabla f(x)$.\\

\textbf{0 - шаг:}
\begin{gather}
y_0 = u_0 = x_0\\
\alpha_0 = 0\\
A_0 = \alpha_0
\end{gather}

\textbf{$\boldsymbol{k+1}$ - шаг:}
\begin{gather}
\alpha_{k+1} = \frac{1}{2L} + \sqrt{\frac{1}{4L^2} + \alpha_k^2}\\
A_{k+1} = A_k + \alpha_{k+1}\\
y_{k+1} = \frac{\alpha_{k+1}u_k + A_k x_k}{A_{k+1}} \label{eqymir} \\
\phi_{k+1}(x) = V(x, u_k) + \alpha_{k+1}[f(y_{k+1}) + \langle \nabla f(y_{k+1}), x - y_{k+1} \rangle]\\
u_{k+1} = \argmin_{x \in Q}\phi_{k+1}(x) \label{equmir}\\
x_{k+1} = \frac{\alpha_{k+1}u_{k+1} + A_k x_k}{A_{k+1}} \label{eqxmir}
\end{gather}
\end{mdframed}

Для данного алгоритма имеется следующая гарантия скорости сходимости.

\begin{theorem}
	\label{sec1:mainTheorem}
	Пусть $x_*$ - решение (\ref{mainTask}), тогда для $x_N$ из алгоритма ЗМТ верно
	\begin{equation*}
	f(x_N) - f(x_*) \leq \frac{4LR^2}{(N+1)^2}
	\end{equation*}
\end{theorem}

Данная теорема утверждает, что предложенный метод является быстро градиентным, проксимальным, причем на каждом шаге считается только одна проекция. Доказательство Теоремы \ref{sec1:mainTheorem} описано в Аппендиксе \ref{appendix:profMMT}.

\section{Решение задачи минимакса с помощью адаптивного ЗМТ} \label{sec:minmax}

Рассмотрим следующую минимаксную задачу оптимизации:
\begin{gather}
	\label{minmax}
	f(x) = \max_{i = 1,\dots,M}\{f_i(x)\} + h(x) \rightarrow \min_{x \in Q}.
\end{gather}

$f_i(x),i = 1,\dots,M$ - выпуклые функции с $L$ липшицевым градиентом на $Q$.
$h(x)$ - выпуклая функция на $Q$. $Q$ - выпуклое, замкнутое множество. Аналогично, как и в разделе \ref{sec:mmt}, мы предполагаем, что $f(x)$ достигает своего минимума в точке $x_*$ и $V(x_*, x_0) \leq R^2$.

Рассмотрим модифицированный алгоритм зеркального метода треугольника для задачи минимакса (\ref{minmax}) с адаптивным подбором "локальной"\,\,константы Липшица:

\begin{mdframed}
\textbf{Дано:} $x_0$ - начальная точка, $N$ - количество шагов и $L_0$ некоторая константа, которая удовлетворяет условию: $L_0 \leq L$.

\textbf{0 - шаг:}
\begin{gather}
y_0 = u_0 = x_0 \\
L_1 = \frac{L_0}{2}\\
\alpha_0 = 0\\
A_0 = \alpha_0
\end{gather}

\textbf{$\boldsymbol{k+1}$ - шаг:}
\begin{center}
	Найти наибольший корень $\alpha_{k+1} : A_k + \alpha_{k+1} = L_{k+1}\alpha^2_{k+1}$
\end{center}
\begin{gather}
A_{k+1} = A_k + \alpha_{k+1}\\
y_{k+1} = \frac{\alpha_{k+1}u_k + A_k x_k}{A_{k+1}} \label{eqymir2} \\
\phi_{k+1}(x) = V(x, u_k) + \alpha_{k+1}\Big(\max_{j = 1,\dots,M}[f_{j}(y_{k+1}) + \langle \nabla f_{j}(y_{k+1}), x - y_{k+1} \rangle] + h(x)\Big)\\
u_{k+1} = \argmin_{x \in Q}\phi_{k+1}(x) \label{equmir2}\\
x_{k+1} = \frac{\alpha_{k+1}u_{k+1} + A_k x_k}{A_{k+1}}. \label{eqxmir2}
\end{gather}
Если выполнено условие
\begin{equation}
\begin{gathered}
f(x_{k+1}) \leq \max_{j = 1,\dots,M}\{f_{j}(y_{k+1}) + \langle \nabla f_{j}(y_{k+1}), x_{k+1} - y_{k+1} \rangle\} +\\  + \frac{L_{k+1}}{2}\norm{x_{k+1} - y_{k+1}}^2 + h(x_{k+1}),
\label{exitL}
\end{gathered}
\end{equation}

то
\begin{gather}
L_{k+2} = \frac{L_{k+1}}{2}
\end{gather}
и перейти к следующему шагу, иначе
\begin{gather}
L_{k+1} = 2L_{k+1}
\end{gather}
и повторить текущий шаг.

\end{mdframed}

\begin{remark}
\leavevmode
\label{remark_maxmin}
\begin{enumerate}
	\item 
	Количество внутренних циклов в каждом шаге конечно. Это следует из того, что на каждом шаге цикла мы увеличиваем $L_{k+1}$ в 2 раза, а значит $L_{k+1}$ через конечное количество шагов станет больше $L$, поэтому из $L$ - Липшевости $\nabla f_j(x)$ следует, что через конечное количество внутренних циклов выполнится условие (\ref{exitL}).
	\item
	Для всех $k \geq 0$ выполнено
	\begin{equation*}
		L_{k} \leq 2L.
	\end{equation*}
	Для $k = 0$ верно из условия на $L_0$. Для $k \geq 1$ это следует из того, что мы выйдем из внутреннего цикла, где подбирается $L_k$, ранее, чем $L_{k}$ станет больше $2L$.
	\item
	Оценим общее число обращений за значениями всех сразу $M$ функций. Внутри каждого шага алгоритм как минимум 1 раз решает задачу (\ref{equmir2}) и делает проверку (\ref{exitL}), пусть $j_k$ - количество дополнительных внутренних циклов $k$-ого шага, где подбирается $L_k$, и, соответственно, количество дополнительных решений (\ref{equmir2}) и проверок (\ref{exitL}). Тогда общее количество обращений за значениями всех функций $f_j(x)$ равно
	\begin{gather*}
	\sum_{k=1}^{N}2(j_k + 1) = \sum_{k=1}^{N}2((j_k - 1) + 2) = \sum_{k=1}^{N}2\left(\log_2(\frac{L_k}{L_{k-1}}) + 2\right) = \\= 4N + 2\log_2\left(\frac{L_N}{L_0}\right) \leq  4N + 2\log_2\left(\frac{2L}{L_0}\right).
	\end{gather*}
	Второе равенство следует из того, что $L_{k} = 2^{j_k}\frac{L_{k-1}}{2}$.
	Поэтому мы получаем, что в среднем на каждом шаге мы будем считать значение всех функций 4 раза. Можно показать, что градиент всех функций $f_j(x)$ мы в среднем будем считать на каждом шаге 2 раза.
\end{enumerate}
\end{remark}

\begin{lemma}
	\label{lemma_maxmin_1}
	Пусть для последовательности $\alpha_k$ выполнено
	\begin{align*}
	\alpha_0 = 0,\\
	A_k = \sum_{i = 0}^{k}\alpha_i,\\
	A_k = L_{k}\alpha_k^2,
	\end{align*}
	где $L_k \leq 2L\,\,\,\forall k\geq0$.
	Тогда верно следующее $\forall k \geq 1$ неравенство:
	 \begin{align}
	 \label{lemma_maxmin_1_1}
	 A_k \geq \frac{(k+1)^2}{8L}.
	 \end{align}
\end{lemma}

\begin{proof}
	Пусть $k = 1$.
	\begin{equation*}
	\alpha_1 = L_{1}\alpha_1^2 
	\end{equation*}
	\begin{equation*}
	A_1 = \alpha_1 = \frac{1}{L_1} \geq \frac{1}{2L}
	\end{equation*}
	Пусть $k \geq 2$, тогда
	\begin{equation*}
	L_{k+1}\alpha^2_{k+1} = A_{k+1}
	\end{equation*}
	\begin{equation*}
	L_{k+1}\alpha^2_{k+1} = A_{k} + \alpha_{k+1}
	\end{equation*}
	\begin{equation*}
	L_{k+1}\alpha^2_{k+1} - \alpha_{k+1} - A_{k} = 0 
	\end{equation*}
	Решая данное квадратное уравнение будем брать наибольший корень, поэтому
	\begin{equation*} 
	\alpha_{k+1} = \frac{1 + \sqrt{\uprule 1 + 4L_{k+1}A_{k}}}{2L_{k+1}}
	\end{equation*}
	По индукции, пусть неравенство (\ref{lemma_maxmin_1_1}) верно для $k$, тогда:
	\begin{gather*}
	\alpha_{k+1} = \frac{1}{2L_{k+1}} + \sqrt{\frac{1}{4L_{k+1}^2} + \frac{A_{k}}{L_{k+1}}} \geq 
	\frac{1}{2L_{k+1}} + \sqrt{\frac{A_{k}}{L_{k+1}}} \geq \\\geq
	\frac{1}{4L} + \frac{1}{\sqrt{2L}}\frac{k+1}{2\sqrt{2L}} =
	\frac{k+2}{4L}
	\end{gather*}
Последнее неравенство следует из $A_k \geq \frac{(k+1)^2}{8L}$, поэтому
	\begin{equation*}
	\alpha_{k+1} \geq \frac{k+2}{4L}
	\end{equation*}
и
	\begin{equation*}
	A_{k+1} = A_k + \alpha_{k+1} = \frac{(k+1)^2}{8L} + \frac{k+2}{4L} \geq \frac{(k+2)^2}{8L}
	\end{equation*}
\end{proof}

\begin{lemma}
	Пусть $\psi(x)$ выпуклая функция и 
	\begin{equation*}
	y = \argmin_{x \in Q}\{\psi(x) + V(x,z)\}
	\end{equation*}
	Тогда

	\begin{equation*}
	\psi(x) + V(x,z) \geq \psi(y) + V(y,z) + V(x,y) \,\,\, \forall x \in Q.
	\end{equation*}
	\label{lemma_maxmin_2}
\end{lemma}

\begin{proof}
	
	По критерию оптимальности:
	\begin{gather*}
		\exists g \in \partial\psi(y), \,\,\, \langle g + \nabla_y V(y, z), x - y  \rangle \geq 0 \,\,\, \forall x \in Q
	\end{gather*}
	Тогда неравенство
	\begin{gather*}
		\psi(x) - \psi(y) \geq \langle g, x - y  \rangle \geq \langle \nabla_y V(y, z), y - x  \rangle
	\end{gather*}
и равенство
	\begin{gather*}
	\langle \nabla_y V(y, z), y - x  \rangle = \langle \nabla d(y) - \nabla d(z), y - x  \rangle = d(y) - d(z) - \langle \nabla d(z), y - z  \rangle +\\ + d(x) - d(y) - \langle \nabla d(y), x - y  \rangle - d(x) + d(z) + \langle \nabla d(z), x - z  \rangle = \\=
	V(y,z) + V(x,y) - V(x,z)
	\end{gather*}
завершают доказательство.
	
\end{proof}
\begin{lemma}
	$\forall x \in Q$ выполнено
	\begin{equation*}
		A_{k+1} f(x_{k+1}) - A_{k} f(x_{k}) + V(x, u_{k+1}) - V(x, u_{k}) \leq \alpha_{k+1}f(x).
	\end{equation*}
	\label{lemma_maxmin_3}
\end{lemma}
\begin{proof}
	Введем обозначение: $l^j_{f}(x;y) = f_{j}(y) + \langle \nabla f_{j}(y), x - y \rangle$.
	\begin{gather*}
	f(x_{k+1}) \leq_{{\tiny \circled{1}}} \max_{j = 1,\dots,M}\{l^j_{f}(x_{k+1};y_{k+1})\}  + \frac{L_{k+1}}{2}\norm{x_{k+1} - y_{k+1}}^2 + h(x_{k+1}) = \\=
	\max_{j = 1,\dots,M}\{l^j_{f}(\frac{\alpha_{k+1}u_{k+1} + A_k x_k}{A_{k+1}};y_{k+1})\}  + \frac{L_{k+1}}{2}\norm{\frac{\alpha_{k+1}u_{k+1} + A_k x_k}{A_{k+1}} - y_{k+1}}^2 + \\+ h(\frac{\alpha_{k+1}u_{k+1} + A_k x_k}{A_{k+1}}) \leq \\\leq
	\max_{j = 1,\dots,M}\{f_{j}(y_{k+1}) + 
	\frac{\alpha_{k+1}}{A_{k+1}}\langle \nabla f_{j}(y_{k+1}), u_{k+1} - y_{k+1} \rangle\ + \\+
	 \frac{A_k}{A_{k+1}}\langle \nabla f_{j}(y_{k+1}), x_k - y_{k+1} \rangle\}  + \frac{L_{k+1} \alpha^2_{k+1}}{2 A^2_{k+1}}\norm{u_{k+1} - u_k}^2 +\\
	 + \frac{\alpha_{k+1}}{A_{k+1}}h(u_{k+1}) + \frac{A_k}{A_{k+1}}h(x_k) \leq \\\leq
	 \frac{A_k}{A_{k+1}}(\max_{j = 1,\dots,M}\{f_{j}(y_{k+1}) + \langle \nabla f_{j}(y_{k+1}), x_k - y_{k+1} \rangle\} + h(x_k))
	 + \\+
	 \frac{\alpha_{k+1}}{A_{k+1}}(\max_{j = 1,\dots,M}\{f_{j}(y_{k+1}) + 
	 \langle \nabla f_{j}(y_{k+1}), u_{k+1} - y_{k+1} \rangle\} + h(u_{k+1})) +\\
	   + \frac{L_{k+1} \alpha^2_{k+1}}{2 A^2_{k+1}}\norm{u_{k+1} - u_k}^2 =_{{\tiny \circled{2}}} \\ =
	 \frac{A_k}{A_{k+1}}(\max_{j = 1,\dots,M}\{l^j_{f}(x_k;y_{k+1})\} + h(x_k))
	 + \\+
	 \frac{\alpha_{k+1}}{A_{k+1}}(\max_{j = 1,\dots,M}\{l^j_{f}(u_{k+1};y_{k+1})\}
	 + \frac{1}{2 \alpha_{k+1}}\norm{u_{k+1} - u_k}^2 + h(u_{k+1})) \leq \\\leq
	 \frac{A_k}{A_{k+1}}(\max_{j = 1,\dots,M}\{l^j_{f}(x_k;y_{k+1})\} + h(x_k))
	 + \\+
	 \frac{\alpha_{k+1}}{A_{k+1}}(\max_{j = 1,\dots,M}\{l^j_{f}(u_{k+1};y_{k+1})\}
	 + \frac{1}{\alpha_{k+1}}V(u_{k+1}, u_k) + h(u_{k+1})) \leq_{{\tiny \circled{3}}} \\
	 \leq \frac{A_k}{A_{k+1}} f(x_k) + \\+
	 \frac{\alpha_{k+1}}{A_{k+1}}(\max_{j = 1,\dots,M}\{l^j_{f}(x;y_{k+1})\} + h(x)
	 + \frac{1}{\alpha_{k+1}}V(x, u_k) - \frac{1}{\alpha_{k+1}}V(x, u_{k+1})) \leq_{{\tiny \circled{4}}} \\\leq
	 \frac{A_k}{A_{k+1}} f(x_k) + 
	 \frac{\alpha_{k+1}}{A_{k+1}} f(x)
	 + \frac{1}{A_{k+1}}V(x, u_k) - \frac{1}{A_{k+1}}V(x, u_{k+1}))
	\end{gather*}
\end{proof}

{\small \circled{1}} - из условия (\ref{exitL})

{\small \circled{2}} - из $A_k = L_{k}\alpha^2_k$

{\small \circled{3}} - из леммы \ref{lemma_maxmin_2} с 
$\psi(x) = \alpha_{k+1}(\max_{j = 1,\dots,M}\{f_{j}(y_{k+1}) + 
\langle \nabla f_{j}(y_{k+1}), x - y_{k+1} \rangle\} + h(x))$ и выпуклость $f_j(x)\,\, \forall j$

{\small \circled{4}} - выпуклость $f_j(x)\,\, \forall j$

\begin{theorem}
	Пусть $x_*$ - решение задачи (\ref{minmax}).
	\begin{equation*}
	f(x_N) - f(x_*) \leq \frac{8LR^2}{(N+1)^2}
	\end{equation*}
\end{theorem}
\begin{proof}
	
	Просуммируем неравенство из леммы \ref{lemma_maxmin_3} по $k = 0, ..., N - 1$
	
	\begin{gather*}
		A_{N} f(x_N) - A_{0} f(x_0) + V(x, u_N) - V(x, u_0) \leq (A_N - A_0)f(x) \\
		A_{N} f(x_N) + V(x, u_N) - V(x, u_0) \leq A_Nf(u)
	\end{gather*}
	
	Возьмем $x = x_*$.
	\begin{gather*}
		A_{N} (f(x_N) - f_*) \leq R^2 \\
		f(x_N) - f_* \leq \frac{R^2}{A_{N}} \leq_{{\tiny \circled{1}}} \frac{8LR^2}{(N+1)^2}
	\end{gather*}
	
	{\small \circled{1}} - из Леммы \ref{lemma_maxmin_1}.
	
\end{proof}

\begin{remark}
\leavevmode

С точностью до константы скорость сходимости никак не изменилась по сравнению с Теоремой \ref{sec1:mainTheorem}. Но, во-первых, данный метод удобен тем, что не надо знать истинную константу Липшица, так как во время работы алгоритма она подбирается автоматически. Во-вторых, при выборе шага $\alpha_k$ мы учитываем некоторую "локальную"\, константу Липшица $L_k$, которая на практике может по ходу работы алгоритма постепенно уменьшаться, из-за чего метод будет на практике работать быстрее.
\end{remark}

\begin{remark}
\leavevmode

Можно заметить, что оптимизационная задача вообще говоря является негладкой, но из-за того, что мы используем структуру задачи, нам удалось получить быструю оценку скорости сходимости градиентного метода.

\end{remark}

\begin{remark}
\leavevmode

Пусть у нас имеется задача с функциональными ограничениями
\begin{gather*}
	f(x) \rightarrow \min_{x \in Q} \\
	\text{при } g_j(x) \leq 0 \,\,\, \forall j = 1,\dots,k
\end{gather*}

Предположим, что мы знаем $f_*$. Тогда можно записать эквивалентную задачу \cite{polyak1983vvedenie}:
\begin{gather*}
	\max\{f(x) - f_*, g_1(x), ..., g_K(x)\} \rightarrow \min_{x \in Q}
\end{gather*}
Для ее решения можно использовать адаптивный алгоритм ЗМТ для задачи минимакса.
\end{remark}

\begin{remark}
\leavevmode

На каждом шаге алгоритма решаем вспомогательную задачу:
	\begin{gather*}
	\phi_{k+1}(x) = V(x, u_k) + \alpha_{k+1}(\max_{j = 1,\dots,M}[f_{j}(y_{k+1}) + \langle \nabla f_{j}(y_{k+1}), x - y_{k+1} \rangle] + h(x))\\
	u_{k+1} = \argmin_{x \in Q}\phi_{k+1}(x)
	\end{gather*}
Пусть $V(x,y) = \frac{1}{2}\norm{x - y}^2_2$, $Q = R^n$ и $h(x) = 0$. Тогда вспомогательную задачу можно свести к задаче квадратичного программирования \cite{nesterov2010methods}:
	\begin{gather*}
	\argmin_{x, t}\{t + \frac{1}{2}\norm{x - u_k}^2_2\} \\
	f_{j}(y_{k+1}) + \langle \nabla f_{j}(y_{k+1}), x - y_{k+1} \rangle \leq t \,\,\, \forall j
	\end{gather*}
	
	Если количество $f_j(x)$ мало, и размерность пространства не очень большая, то задачу можно решить быстро методом внутренней точки.
\end{remark}

\section{Зеркальный метод треугольника с неточным $(\delta, L)$-оракулом} \label{sec:mmtDL}

В данном разделе будем решать следующую задачу:
\begin{align}
\label{mainTask3}
F(x) \stackrel{def}{=} f(x) + h(x) \rightarrow \min_{x \in Q}
\end{align}

Условия на $f(x)$ такие же, как и в разделе \ref{sec:mmt}. Будем предполагать, что $x_*$ решение (\ref{mainTask3}) и $V(x_*, x_0) \leq R^2$. В отличие от раздела \ref{sec:mmt}, нам будем доступен только $(\delta, L)$-оракул \cite{devolder2013exactness}. $h(x)$ - выпуклая функция на $Q$.

\begin{definition}
\leavevmode

$(\delta, L)$-оракулом будем называть оракул, который на запрашиваемую точку $y$ дает пару $(f_\delta(y), \nabla f_\delta(y))$ такую, что
\begin{gather}
\label{exitLDLOrig}
0 \leq f(x) - f_\delta(y) - \langle\nabla f_\delta(y), x - y\rangle \leq \frac{L}{2}\norm{x - y}^2 + \delta \,\,\, \forall x \in Q.
\end{gather}

\end{definition}

\begin{corollary}

Возьмем $x = y$ в (\ref{exitLDLOrig}), тогда
\begin{gather}
\label{exitLDLOrig2}
f_\delta(y) \leq f(y) \leq f_\delta(y) + \delta \,\,\,\forall y \in Q.
\end{gather}
\end{corollary}

Рассмотрим алгоритм зеркального метода треугольника с неточным $(\delta, L)$-оракулом.

\begin{mdframed}
\textbf{Дано:} $x_0$ - начальная точка, $N$ - количество шагов, $\delta$ и $L_0$ некоторая константа, которая удовлетворяет условию: $L_0 \leq L$.

\textbf{0 - шаг:}
\begin{gather}
y_0 = u_0 = x_0 \\
L_1 = \frac{L_0}{2}\\
\alpha_0 = 0\\
A_0 = \alpha_0
\end{gather}

\textbf{$\boldsymbol{k+1}$ - шаг:}
\begin{center}
	Найти наибольший корень $\alpha_{k+1} : A_k + \alpha_{k+1} = L_{k+1}\alpha^2_{k+1}$
\end{center}
\begin{gather}
A_{k+1} = A_k + \alpha_{k+1}\\
y_{k+1} = \frac{\alpha_{k+1}u_k + A_k x_k}{A_{k+1}} \label{eqymir2DL} \\
\phi_{k+1}(x) = V(x, u_k) + \alpha_{k+1}(f_\delta(y_{k+1}) + \langle \nabla f_\delta(y_{k+1}), x - y_{k+1} \rangle + h(x))\\
u_{k+1} = \argmin_{x \in Q}\phi_{k+1}(x) \label{equmir2DL}\\
x_{k+1} = \frac{\alpha_{k+1}u_{k+1} + A_k x_k}{A_{k+1}} \label{eqxmir2DL}
\end{gather}
Если выполнено условие
\begin{equation}
\begin{gathered}
f_\delta(x_{k+1}) \leq f_\delta(y_{k+1}) + \langle \nabla f_\delta(y_{k+1}), x_{k+1} - y_{k+1} \rangle\ +\\  + \frac{L_{k+1}}{2}\norm{x_{k+1} - y_{k+1}}^2 + \delta,
\label{exitLDL}
\end{gathered}
\end{equation}

то
\begin{gather}
L_{k+2} = \frac{L_{k+1}}{2}
\end{gather}
и перейти к следующему шагу, иначе
\begin{gather}
L_{k+1} = 2L_{k+1}
\end{gather}
и повторить текущий шаг.
\end{mdframed}

\begin{remark}
\leavevmode

Все свойства из Замечания \ref{remark_maxmin} сохраняются, но надо отметить, что из (\ref{exitLDLOrig}) и (\ref{exitLDLOrig2}) следует (\ref{exitLDL}) с $L_{k+1} \geq L$. Данное замечание гарантирует, что через конечное количество внутренних циклов при подборе $L_k$ будет выполнено (\ref{exitLDL}).

\end{remark}

Докажем основную лемму, которая практически полностью повторяет лемму \ref{lemma_maxmin_3}.

\begin{lemma}
	$\forall x \in Q$ выполнено
	\begin{equation*}
		A_{k+1} F(x_{k+1}) - A_{k} F(x_{k}) + V(x, u_{k+1}) - V(x, u_{k}) \leq \alpha_{k+1}F(x) + 2\delta A_{k+1}
	\end{equation*}
	\label{lemma_maxmin_3DL}
\end{lemma}
\begin{proof}
	Введем обозначение: $l_f^\delta(x;y) = f_\delta(y) + \langle \nabla f_\delta(y), x - y \rangle$.
	\begin{gather*}
	F(x_{k+1}) \leq_{{\tiny \circled{1}}} l_f^\delta(x_{k+1};y_{k+1})  + \frac{L_{k+1}}{2}\norm{x_{k+1} - y_{k+1}}^2 + h(x_{k+1}) + 2\delta = \\=
	l_f^\delta(\frac{\alpha_{k+1}u_{k+1} + A_k x_k}{A_{k+1}};y_{k+1})  + \frac{L_{k+1}}{2}\norm{\frac{\alpha_{k+1}u_{k+1} + A_k x_k}{A_{k+1}} - y_{k+1}}^2 + \\+ h(\frac{\alpha_{k+1}u_{k+1} + A_k x_k}{A_{k+1}}) + 2\delta\leq \\\leq
	f_\delta(y_{k+1}) + 
	\frac{\alpha_{k+1}}{A_{k+1}}\langle \nabla f_\delta(y_{k+1}), u_{k+1} - y_{k+1} \rangle\ + \\+
	 \frac{A_k}{A_{k+1}}\langle \nabla f_\delta(y_{k+1}), x_k - y_{k+1} \rangle  + \frac{L_{k+1} \alpha^2_{k+1}}{2 A^2_{k+1}}\norm{u_{k+1} - u_k}^2 +\\
	 + \frac{\alpha_{k+1}}{A_{k+1}}h(u_{k+1}) + \frac{A_k}{A_{k+1}}h(x_k) + 2\delta= \\=
	 \frac{A_k}{A_{k+1}}(f_\delta(y_{k+1}) + \langle \nabla f_\delta(y_{k+1}), x_k - y_{k+1} \rangle + h(x_k))
	 + \\+
	 \frac{\alpha_{k+1}}{A_{k+1}}(f_\delta(y_{k+1}) + 
	 \langle \nabla f_\delta(y_{k+1}), u_{k+1} - y_{k+1} \rangle + h(u_{k+1}))+ \\
	   + \frac{L_{k+1} \alpha^2_{k+1}}{2 A^2_{k+1}}\norm{u_{k+1} - u_k}^2 + 2\delta=_{{\tiny \circled{2}}} \\ =
	 \frac{A_k}{A_{k+1}}(l_f^\delta(x_k;y_{k+1}) + h(x_k))
	 + \\+
	 \frac{\alpha_{k+1}}{A_{k+1}}(l_f^\delta(u_{k+1};y_{k+1})
	 + \frac{1}{2 \alpha_{k+1}}\norm{u_{k+1} - u_k}^2 + h(u_{k+1})) + 2\delta\leq \\\leq
	 \frac{A_k}{A_{k+1}}(l_f^\delta(x_k;y_{k+1}) + h(x_k))
	 + \\+
	 \frac{\alpha_{k+1}}{A_{k+1}}(l_f^\delta(u_{k+1};y_{k+1})
	 + \frac{1}{\alpha_{k+1}}V(u_{k+1}, u_k) + h(u_{k+1})) + 2\delta\leq_{{\tiny \circled{3}}} \\\leq
	 \frac{A_k}{A_{k+1}} F(x_k) + \\+
	 \frac{\alpha_{k+1}}{A_{k+1}}(l_f^\delta(x;y_{k+1}) + h(x)
	 + \frac{1}{\alpha_{k+1}}V(x, u_k) - \frac{1}{\alpha_{k+1}}V(x, u_{k+1})) + 2\delta \leq_{{\tiny \circled{4}}} \\\leq
	 \frac{A_k}{A_{k+1}} F(x_k) + 
	 \frac{\alpha_{k+1}}{A_{k+1}} F(x)
	 + \frac{1}{A_{k+1}}V(x, u_k) - \frac{1}{A_{k+1}}V(x, u_{k+1})) + 2\delta
	\end{gather*}
\end{proof}

{\small \circled{1}} - из условия (\ref{exitLDL}) и (\ref{exitLDLOrig2}) 

{\small \circled{2}} - из $A_k = L_{k}\alpha^2_k$

{\small \circled{3}} - из леммы \ref{lemma_maxmin_2} с 
$\psi(x) = \alpha_{k+1}(f_\delta(y_{k+1}) + 
\langle \nabla f_\delta(y_{k+1}), x - y_{k+1} \rangle + h(x))$ и левая часть (\ref{exitLDLOrig})

{\small \circled{4}} - левая часть (\ref{exitLDLOrig})

\begin{theorem}
	\label{mainTheoremDL}
	Пусть $x_*$ - решения задачи (\ref{mainTask3}), тогда
	\begin{equation*}
	F(x_N) - F(x_*) \leq \frac{8LR^2}{(N+1)^2}+ 2N\delta
	\end{equation*}
\end{theorem}
\begin{proof}
	
	Просуммируем нер-во из леммы \ref{lemma_maxmin_3DL} по $k = 0, ..., N - 1$
	
	\begin{gather*}
		A_{N} F(x_N) - A_{0} F(x_0) + V(x, u_N) - V(x, u_0) \leq (A_N - A_0)F(x) + 2\delta\sum_{k = 0}^{N-1}A_{k+1} \\
		A_{N} F(x_N) + V(x, u_N) - V(x, u_0) \leq A_NF(x) + 2\delta\sum_{k = 0}^{N-1}A_{k+1}
	\end{gather*}
	
	Возьмем $x = x_*$ и используем, что $\forall k = 1,\dots,N$ $A_{k} \leq A_N$, так как по определению $A_{k}$ - неубывающая последовательность.
	
	\begin{gather*}
		A_{N} (F(x_N) - F_*) \leq R^2 + 2NA_N\delta \\
		F(x_N) - F_* \leq \frac{R^2}{A_{N}} + 2N\delta \leq_{{\tiny \circled{1}}} \frac{8LR^2}{(N+1)^2} + 2N\delta
	\end{gather*}
		
	{\small \circled{1}} - из Леммы \ref{lemma_maxmin_1}.
	
\end{proof}

\begin{remark}
Пусть $f(y)$ - $\alpha$-квази-выпуклая функция, то есть $f(y)$ - невыпуклая функция, но верно свойство $\langle\nabla f(y), y - x_*\rangle \geq \alpha(f(y) - f(x_*))$  $\forall y \in Q$. В Лемме \ref{lemma_maxmin_3DL} единственное место, где используется выпуклость, - это переход \circled{4}. Вместо выпуклости достаточно потребовать $\alpha$-квази-выпуклости c $\alpha = 1$ и вместо условия $\forall x \in Q$ потребовать $x = x_*$ в Лемме \ref{lemma_maxmin_3DL}, чтобы выполнялся переход \circled{4}.

$\alpha$-квази-выпуклые функции имеют, в частности, следующее приложение \cite{hardt2017identity}.
\end{remark}

\begin{remark}
На подобии \cite{gasnikov2016universal} можно получить универсальный метод, если в алгоритме ЗМТ на месте $\delta$ поставить $\frac{\alpha_{k+1}}{A_{k+1}}\epsilon$ в (\ref{exitLDL}).
\end{remark}

\section{Зеркальный метод треугольника с неточным выборочным стохастическим $(\delta, L)$-оракулом} \label{sec:mmtDLST}

Будем считать, что мы решаем задачу (\ref{mainTask3}). Ограничимся случаем, когда выбранная норма является евклидовой.

\begin{definition}
\leavevmode

Стохастическим $(\delta, L)$-оракулом будем называть оракул, который на запрашиваемую точку $y$ дает пару $(f_\delta(y), \nabla f_\delta(y;\xi))$ такую, что
\begin{gather}
\label{exitLDLSTOrig}
0 \leq f(x) - f_\delta(y) - \langle\nabla f_\delta(y), x - y\rangle \leq \frac{L}{2}\norm{x - y}^2 + \delta \,\,\, \forall x \in Q.
\end{gather}
\begin{gather}
\label{ST1}
\mathbb{E}\nabla f_\delta(y;\xi) = \nabla f_\delta(y)\,\,\, \forall y \in Q.
\end{gather}
\begin{gather}
\label{ST2}
\mathbb{E}\exp\Bigg(\frac{\norm{\nabla f_\delta(y;\xi) - \nabla f_\delta(y)}^2_*}{D}\Bigg) \leq \exp(1)\,\,\, \forall y \in Q.
\end{gather}

\end{definition}

\begin{definition}
\leavevmode
Определим константу $D_Q$ такую, что
\begin{gather}
D_Q \geq \max_{x,y \in Q}\norm{x - y}
\end{gather}
\end{definition}
Считаем, что $D_Q < \infty$.

В методе, который мы предложим далее, будем оценивать истинный градиент на каждом шаге с помощью некоторого количества $\nabla f_\delta(y;\xi_j)\,\, j \in [1\dots m_{k+1}]$, используя технику mini-batch.

\begin{definition}
\leavevmode
\begin{gather}
\widetilde{\nabla}^{m_{k+1}} f_\delta(y) = \frac{1}{m_{k+1}}\sum_{j=1}^{m_{k+1}}\nabla f_\delta(y;\xi_j)
\end{gather}
\end{definition}

Приведем важные следствия:

\begin{corollary}
\label{corST2}

Пусть $(f_\delta(y), \nabla f_\delta(y;\xi_i)), i = 1,\dots,m_{k+1}\,$ - $m_{k+1}$ независимых выхода стохастического $(\delta, L)$-оракула, $x, y \in Q$ - случайные векторы, $y$ и $\xi_i\,\,i = 1,\dots,m_{k+1}$ - независимы, $\widetilde{L}$ случайная константа, такая, что $\widetilde{L} \geq \frac{3}{2}L$ и выбрано произвольное $\Omega \geq \sqrt{2} - 1$, тогда
\begin{gather*}
\mathbb{P}\Bigg(f_\delta(x) - f_\delta(y) - \langle\widetilde{\nabla}^{m_{k+1}} f_\delta(y), x - y\rangle >\\ > (1 + 2\Omega+ \Omega^2)\frac{3D}{\widetilde{L}m_{k+1}} + \frac{\widetilde{L}}{2}\norm{x - y}^2 + \delta\Bigg) \leq \exp(-\Omega^2/2).
\end{gather*}
\end{corollary}

\begin{proof}
Рассмотрим правую часть условия (\ref{exitLDLSTOrig}).
\begin{gather*}
f(x) - f_\delta(y) - \langle\nabla f_\delta(y), x - y\rangle \leq \frac{L}{2}\norm{x - y}^2 + \delta \\
\end{gather*}
Учтем (\ref{exitLDLOrig2}), тогда
\begin{gather*}
f_\delta(x) - f_\delta(y) - \langle\nabla f_\delta(y), x - y\rangle \leq \frac{L}{2}\norm{x - y}^2 + \delta,\\
f_\delta(x) - f_\delta(y) - \langle\widetilde{\nabla}^{m_{k+1}} f_\delta(y), x - y\rangle \leq\\ \leq \langle\nabla f_\delta(y) - \widetilde{\nabla}^{m_{k+1}} f_\delta(y), x - y\rangle + \frac{L}{2}\norm{x - y}^2 + \delta\leq\\ \leq \langle\nabla f_\delta(y) - \widetilde{\nabla}^{m_{k+1}} f_\delta(y), x - y\rangle + \frac{\widetilde{L}}{3}\norm{x - y}^2 + \delta
\end{gather*}
Воспользуемся неравенством Фенхеля \cite{devolder2013exactness} (формула (7.6))
\begin{equation}
\begin{gathered}
\label{ST_H1}
f_\delta(x) - f_\delta(y) - \langle\widetilde{\nabla}^{m_{k+1}} f_\delta(y), x - y\rangle \leq\\ \leq \frac{\widetilde{L}}{6}\norm{x - y}^2 + \frac{3}{\widetilde{L}}\norm{\widetilde{\nabla}^{m_{k+1}} f_\delta(y) - \nabla f_\delta(y)}^2_* + \frac{\widetilde{L}}{3}\norm{x - y}^2 + \delta
\end{gathered}
\end{equation}
Оценим вероятность того, что
\begin{equation}
\begin{gathered}
\label{ST_H2}
f_\delta(x) - f_\delta(y) - \langle\widetilde{\nabla}^{m_{k+1}} f_\delta(y), x - y\rangle >\\ > (1 + 2\Omega+ \Omega^2)\frac{3D}{Lm_{k+1}} + \frac{\widetilde{L}}{2}\norm{x - y}^2 + \delta
\end{gathered}
\end{equation}
Учитывая (\ref{ST_H1}), из (\ref{ST_H2}) будет следовать
\begin{gather}
\label{ST_H3}
\frac{3}{\widetilde{L}}\norm{\widetilde{\nabla}^{m_{k+1}} f_\delta(y) - \nabla f_\delta(y)}^2_* > (1 + 2\Omega+ \Omega^2)\frac{3D}{\widetilde{L}m_{k+1}},
\end{gather}
что эквивалентно
\begin{gather}
\label{ST_H4}
\norm{\widetilde{\nabla}^{m_{k+1}} f_\delta(y) - \nabla f_\delta(y)}^2_* > (1 + 2\Omega+ \Omega^2)\frac{D}{m_{k+1}},
\end{gather}
Воспользуемся следующим фактом \cite{juditsky2008large}, пусть $\gamma_1,\dots,\gamma_N$ - случайные независимые вектора такие, что
\begin{gather*}
\mathbb{E}\Big(\exp\Big(\frac{\norm{\gamma_i}^2}{\sigma^2}\Big)\Big) \leq \exp(1),\,\,\,\,\mathbb{E}\gamma_i = 0,
\end{gather*}
тогда верно для $\forall \Omega \geq \sqrt{2} - 1$
\begin{gather*}
\mathbb{P}\Bigg(\norm{\sum_{i=1}^{N}\gamma_i} \geq (1 + \Omega)\sqrt{N}\sigma\Bigg) \leq \exp(-\Omega^2/2).
\end{gather*}
Возьмем $\gamma_i = \nabla f_\delta(\widetilde{y};\xi_i) - \nabla f_\delta(\widetilde{y})$, где $\widetilde{y}$ - неслучайный вектор, и $\sigma^2 = D$ и учтем (\ref{ST1}) и (\ref{ST2}), тогда
\begin{gather*}
\mathbb{P}\Bigg(\norm{\sum_{j=1}^{m_{k+1}}\Big(\nabla f_\delta(\widetilde{y};\xi_j) - \nabla f_\delta(\widetilde{y})\Big)}_* > (1 + \Omega)\sqrt{m_{k+1}}\sqrt{D}\Bigg) \leq \exp(-\Omega^2/2)\\
\mathbb{P}\Bigg(\norm{\widetilde{\nabla}^{m_{k+1}} f_\delta(\widetilde{y}) - \nabla f_\delta(\widetilde{y})}_* > (1 + \Omega)\frac{\sqrt{D}}{\sqrt{m_{k+1}}}\Bigg) \leq \exp(-\Omega^2/2)
\end{gather*}

 тогда 
 
\begin{gather*}
\mathbb{P}\Bigg(\norm{\widetilde{\nabla}^{m_{k+1}} f_\delta(y) - \nabla f_\delta(y)}_* > (1 + \Omega)\frac{\sqrt{D}}{\sqrt{m_{k+1}}}\Bigg) = \\ =
\mathbb{E}\Bigg[\mathbb{P}\Bigg(\norm{\widetilde{\nabla}^{m_{k+1}} f_\delta(\widetilde{y}) - \nabla f_\delta(\widetilde{y})}_* > (1 + \Omega)\frac{\sqrt{D}}{\sqrt{m_{k+1}}}\Bigg|y = \widetilde{y}\Bigg)\Bigg] \leq\\\leq
\mathbb{E}\exp(-\Omega^2/2) = \exp(-\Omega^2/2)
\end{gather*}

и

\begin{gather*}
\mathbb{P}\Bigg(\frac{3}{\widetilde{L}}\norm{\widetilde{\nabla}^{m_{k+1}} f_\delta(y) - \nabla f_\delta(y)}_*^2 > (1 + 2\Omega+ \Omega^2)\frac{3D}{\widetilde{L}m_{k+1}}\Bigg) \leq \exp(-\Omega^2/2).
\end{gather*}
Из этого неравенства и из того, что из (\ref{ST_H2}) следует (\ref{ST_H3}), получаем утверждение следствия.
\end{proof}

Рассмотрим алгоритм зеркального метода треугольника со стохастическим $(\delta, L)$-оракулом.

\begin{definition}
$\widetilde{\Omega} = 1 + 2\Omega+ \Omega^2$
\end{definition}

\begin{mdframed}
\textbf{Дано:} $x_0$ - начальная точка, $\epsilon$ - желаемая точность решения, $\delta$, $L$ - константа из $(\delta, L)$-оракула, $\beta$ - доверительный уровень.

\iftrue
Возьмем 
\begin{gather*}
N = \left\lceil\frac{2\sqrt{3}\sqrt{L}D_Q}{\sqrt{\epsilon}}\right\rceil\\
\Omega = \sqrt{2\ln{\frac{N}{\beta}}}.
\end{gather*}

\textbf{0 - шаг:}
\begin{gather}
y_0 = u_0 = x_0 \\
L_1 = \frac{L}{2}\\
\alpha_0 = 0\\
A_0 = \alpha_0
\end{gather}

\textbf{$\boldsymbol{k+1}$ - шаг:}
\begin{gather}
	\text{Найти наибольший корень } \alpha_{k+1} : A_k + \alpha_{k+1} = L_{k+1}\alpha^2_{k+1} \label{eqymir2DLSTA}
\end{gather}
\begin{gather}
A_{k+1} = A_k + \alpha_{k+1}\\
y_{k+1} = \frac{\alpha_{k+1}u_k + A_k x_k}{A_{k+1}} \label{eqymir2DLST}
\end{gather}
\begin{gather}
m_{k+1} = \Big\lceil\frac{3D\widetilde{\Omega}\alpha_{k+1}}{\epsilon}\Big\rceil
\label{eqymir2DLSTMK}
\end{gather}

\begin{center}
Сгенерировать: $ \widetilde{\nabla}^{m_{k+1}} f_\delta(y_{k+1})$
\end{center}
\begin{gather}
\phi_{k+1}(x) = V(x, u_k) + \alpha_{k+1}(f_\delta(y_{k+1}) + \langle \widetilde{\nabla}^{m_{k+1}} f_\delta(y_{k+1}), x - y_{k+1} \rangle + h(x))\\
u_{k+1} = \argmin_{x \in Q}\phi_{k+1}(x) \label{equmir2DLST}\\
x_{k+1} = \frac{\alpha_{k+1}u_{k+1} + A_k x_k}{A_{k+1}} \label{eqxmir2DLST}
\end{gather}
Если выполнено условие
\begin{equation}
\begin{gathered}
f_\delta(x_{k+1}) \leq f_\delta(y_{k+1}) + \langle \widetilde{\nabla}^{m_{k+1}} f_\delta(y_{k+1}), x_{k+1} - y_{k+1} \rangle\ +\\  + \frac{L_{k+1}}{2}\norm{x_{k+1} - y_{k+1}}^2 + \frac{3D\widetilde{\Omega}}{L_{k+1}m_{k+1}} + \delta,
\label{exitLDLST}
\end{gathered}
\end{equation}
то
\begin{gather}
L_{k+2} = \frac{L_{k+1}}{2}
\end{gather}
и перейти к следующему шагу, иначе
\begin{gather}
L_{k+1} = 2L_{k+1}
\end{gather}
и повторить текущий шаг.
\end{mdframed}

\begin{lemma}
\leavevmode
\label{remark2ST}
Пусть $B_N$ - событие того, что хотя бы на одном из первых $N$ шагов алгоритма не выполнится условие (\ref{exitLDLST}) для $L_{k+1}$, в то время как $L_{k+1} \geq \frac{3}{2}L$, тогда, если $\Omega = \sqrt{2\ln{\frac{N}{\beta}}}$, то $\mathbb{P}\Big(B_N\Big) \leq \beta$.

\end{lemma}
\label{STL1}
\begin{proof}
С учетом Следствия \ref{corST2} получаем

\begin{gather*}
\mathbb{P}\Big(B_N\Big) \leq_{{\tiny \circled{1}}} \sum_{k = 0}^{N - 1}\mathbb{P}\Bigg(f_\delta(x_{k+1}) - f_\delta(y_{k+1}) - \langle\widetilde{\nabla}^{m_{k+1}} f_\delta(y_{k+1}), x_{k+1} - y_{k+1}\rangle >\\ > \frac{3D\widetilde{\Omega}}{L_{k+1}m_{k+1}} + \frac{L_{k+1}}{2}\norm{x_{k+1} - y_{k+1}}^2 + \delta\Bigg) \leq_{{\tiny \circled{2}}} N\exp(-\Omega^2/2)
\end{gather*}

{\small \circled{1}} - из неравенство Бонферрони для $B_N =  \bigcup\limits_{i=1}^{N}\widetilde{B}_i$, где $\widetilde{B}_i$ - событие того, что на $i$ шаге не выполнилось (\ref{exitLDLST}) при $L_{i} \geq \frac{3}{2}L$.

{\small \circled{2}} - Следствие \ref{corST2}

Так как по условию леммы $\Omega = \sqrt{2\ln{\frac{N}{\beta}}}$, то
\begin{gather*}
\mathbb{P}\Big(B_N\Big) \leq N\exp(-\Omega^2/2) \leq \beta.\\
\end{gather*}

\end{proof}
\begin{definition}
$\widetilde{L}_N = \max\limits_{k = 0\dots N-1}L_{k+1}$
\end{definition}

\begin{lemma}
	\label{lemma_maxmin_DLST1}
	Пусть для последовательности $\alpha_k$ выполнено
	\begin{align*}
	\alpha_0 = 0\\
	A_k = \sum_{i = 0}^{k}\alpha_i\\
	A_k = L_{k}\alpha_k^2,
	\end{align*}
	где $\{L_k\}$ - последовательность, генерируемая алгоритмом.
	
	Тогда c вероятностью $1 - \beta$
	 \begin{align*}
	 A_k \geq \frac{(k+1)^2}{12L},\,\,\,\forall k = 1,\dots,N
	 \end{align*}
\end{lemma}

\begin{proof}
Воспользуемся Леммой \ref{remark2ST}, которая говорит следующая: с вероятностью меньше или равной $\beta$ хоть одно $L_k \geq \frac{3}{2}L$. Так как может возникнуть пороговая ситуация, когда $L_k \in (\frac{3}{4}L,\frac{3}{2}L)$, тогда в силу удвоения $L_k$ в крайнем случаем получаем, что с вероятностью больше или равной $1 - \beta$ все $L_k \leq 3L$. Далее доказывается аналогично Лемме \ref{lemma_maxmin_1}.
\end{proof}

\begin{remark}
Так как с вероятностью $1 - \beta$ выполнено $\widetilde{L}_N \leq 3L$, то, как и в Замечании \ref{remark_maxmin}, можно получить, что с вероятностью $1 - \beta$ в среднем на каждом шаге мы будем считать значение всех функций 4 раза, а стохастического градиента $\widetilde{\nabla}^{m_{k+1}} f_\delta(y_{k+1})$ - 2 раза. 
\end{remark}

Введем обозначение: $l_f^\delta(x;y) = f_\delta(y) + \langle \widetilde{\nabla}^{m_{k+1}} f_\delta(y), x - y \rangle$.
\begin{lemma}
	$\forall x \in Q$ выполнено
	\begin{gather*}
		l_f^\delta(x_{k+1};y_{k+1})  + \frac{L_{k+1}}{2}\norm{x_{k+1} - y_{k+1}}^2 + h(x_{k+1}) \leq\\\leq \frac{A_k}{A_{k+1}}(l_f^\delta(x_k;y_{k+1}) + h(x_k)) + \\+
			 \frac{\alpha_{k+1}}{A_{k+1}}(l_f^\delta(x;y_{k+1}) + h(x)
			 + \frac{1}{\alpha_{k+1}}V(x, u_k) - \frac{1}{\alpha_{k+1}}V(x, u_{k+1}))
	\end{gather*}
	\label{lemma_maxmin_3DLST}
\end{lemma}
\begin{proof}

	\begin{gather*}
	l_f^\delta(x_{k+1};y_{k+1})  + \frac{L_{k+1}}{2}\norm{x_{k+1} - y_{k+1}}^2 + h(x_{k+1}) = \\=
	l_f^\delta(\frac{\alpha_{k+1}u_{k+1} + A_k x_k}{A_{k+1}};y_{k+1})  + \frac{L_{k+1}}{2}\norm{\frac{\alpha_{k+1}u_{k+1} + A_k x_k}{A_{k+1}} - y_{k+1}}^2 + \\+ h(\frac{\alpha_{k+1}u_{k+1} + A_k x_k}{A_{k+1}})\leq \\\leq
	f_\delta(y_{k+1}) + 
	\frac{\alpha_{k+1}}{A_{k+1}}\langle \widetilde{\nabla}^{m_{k+1}} f_\delta(y_{k+1}), u_{k+1} - y_{k+1} \rangle\ + \\+
	 \frac{A_k}{A_{k+1}}\langle \widetilde{\nabla}^{m_{k+1}} f_\delta(y_{k+1}), x_k - y_{k+1} \rangle  + \frac{L_{k+1} \alpha^2_{k+1}}{2 A^2_{k+1}}\norm{u_{k+1} - u_k}^2 +\\
	 + \frac{\alpha_{k+1}}{A_{k+1}}h(u_{k+1}) + \frac{A_k}{A_{k+1}}h(x_k)= \\=
	 \frac{A_k}{A_{k+1}}(f_\delta(y_{k+1}) + \langle \widetilde{\nabla}^{m_{k+1}} f_\delta(y_{k+1}), x_k - y_{k+1} \rangle + h(x_k))
	 + \\+
	 \frac{\alpha_{k+1}}{A_{k+1}}(f_\delta(y_{k+1}) + 
	 \langle \widetilde{\nabla}^{m_{k+1}} f_\delta(y_{k+1}), u_{k+1} - y_{k+1} \rangle + h(u_{k+1}))+ \\
	   + \frac{L_{k+1} \alpha^2_{k+1}}{2 A^2_{k+1}}\norm{u_{k+1} - u_k}^2=_{{\tiny \circled{1}}} \\ =
	 \frac{A_k}{A_{k+1}}(l_f^\delta(x_k;y_{k+1}) + h(x_k))
	 + \\+
	 \frac{\alpha_{k+1}}{A_{k+1}}(l_f^\delta(u_{k+1};y_{k+1})
	 + \frac{1}{2 \alpha_{k+1}}\norm{u_{k+1} - u_k}^2 + h(u_{k+1}))\leq \\\leq
	 \frac{A_k}{A_{k+1}}(l_f^\delta(x_k;y_{k+1}) + h(x_k))
	 + \\+
	 \frac{\alpha_{k+1}}{A_{k+1}}(l_f^\delta(u_{k+1};y_{k+1})
	 + \frac{1}{\alpha_{k+1}}V(u_{k+1}, u_k) + h(u_{k+1}))\leq_{{\tiny \circled{2}}} \\\leq
	 \frac{A_k}{A_{k+1}}(l_f^\delta(x_k;y_{k+1}) + h(x_k)) + \\+
	 \frac{\alpha_{k+1}}{A_{k+1}}(l_f^\delta(x;y_{k+1}) + h(x)
	 + \frac{1}{\alpha_{k+1}}V(x, u_k) - \frac{1}{\alpha_{k+1}}V(x, u_{k+1}))
	\end{gather*}
\end{proof}

{\small \circled{1}} - из $A_k = L_{k}\alpha^2_k$

{\small \circled{2}} - из леммы \ref{lemma_maxmin_2} с 
$\psi(x) = \alpha_{k+1}(f_\delta(y_{k+1}) + 
\langle \widetilde{\nabla}^{m_{k+1}} f_\delta(y_{k+1}), x - y_{k+1} \rangle + h(x))$

\begin{lemma}
	С вероятностью больше или равной $1 - \beta$ $\forall x \in Q$, $\forall k \geq 0$
	\begin{gather*}
		A_{k+1} F(x_{k+1}) - A_{k} F(x_{k}) + V(x, u_{k+1}) - V(x, u_{k}) \leq\\\leq \alpha_{k+1}F(x) + 2\delta A_{k+1} + \frac{3D\widetilde{\Omega}}{L_{k+1}m_{k+1}}A_{k+1}+ \alpha_{k+1}\langle \widetilde{\nabla}^{m_{k+1}} f_\delta(y_{k+1})-\nabla f_\delta(y_{k+1}), x - u_k \rangle
	\end{gather*}
	\label{lemma_maxmin_3DLST_2}
\end{lemma}

\begin{proof}

Надо отметить, что с вероятностью больше или равной $1 - \beta$ выполнится за конечное количество шагов (\ref{exitLDLST}). Это следует из Леммы \ref{remark2ST}.

\begin{gather*}
F(x_{k+1}) \leq_{{\tiny \circled{1}}} l_{f_\delta}(x_{k+1};y_{k+1})  + \frac{L_{k+1}}{2}\norm{x_{k+1} - y_{k+1}}^2 + h(x_{k+1}) + \frac{3D\widetilde{\Omega}}{L_{k+1}m_{k+1}} + 2\delta \leq_{{\tiny \circled{2}}} \\\leq
\frac{A_k}{A_{k+1}}(l_{f_\delta}(x_k;y_{k+1}) + h(x_k)) + \\+
	 \frac{\alpha_{k+1}}{A_{k+1}}(l_{f_\delta}(x;y_{k+1}) + h(x)
	 + \frac{1}{\alpha_{k+1}}V(x, u_k) - \frac{1}{\alpha_{k+1}}V(x, u_{k+1})) + \frac{3D\widetilde{\Omega}}{L_{k+1}m_{k+1}} + 2\delta
\end{gather*}

{\small \circled{1}} - из условия (\ref{exitLDLST}) и (\ref{exitLDLOrig2})

{\small \circled{2}} - из Леммы \ref{lemma_maxmin_3DLST}
\begin{gather*}
 F(x_{k+1}) \leq \frac{A_k}{A_{k+1}}(f_\delta(y_{k+1}) + \langle \widetilde{\nabla}^{m_{k+1}} f_\delta(y_{k+1}), x_k - y_{k+1} \rangle + h(x_k)) + \\ +
\frac{\alpha_{k+1}}{A_{k+1}}(f_\delta(y_{k+1}) + \langle \widetilde{\nabla}^{m_{k+1}} f_\delta(y_{k+1}), x - y_{k+1} \rangle + h(x)
	 +\\+ \frac{1}{\alpha_{k+1}}V(x, u_k) - \frac{1}{\alpha_{k+1}}V(x, u_{k+1})) + \frac{3D\widetilde{\Omega}}{L_{k+1}m_{k+1}} + 2\delta= \\ = \frac{A_k}{A_{k+1}}(f_\delta(y_{k+1}) + \langle \nabla f_\delta(y_{k+1}), x_k - y_{k+1} \rangle+ h(x_k) +\\+ \langle \widetilde{\nabla}^{m_{k+1}} f_\delta(y_{k+1})-\nabla f_\delta(y_{k+1}), x_k - y_{k+1} \rangle) + \\ +
\frac{\alpha_{k+1}}{A_{k+1}}(f_\delta(y_{k+1}) + \langle \nabla f_\delta(y_{k+1}), x - y_{k+1} \rangle + h(x)+\\+ \langle \widetilde{\nabla}^{m_{k+1}} f_\delta(y_{k+1})-\nabla f_\delta(y_{k+1}), x - y_{k+1} \rangle
	 +\\+ \frac{1}{\alpha_{k+1}}V(x, u_k) - \frac{1}{\alpha_{k+1}}V(x, u_{k+1})) + \frac{3D\widetilde{\Omega}}{L_{k+1}m_{k+1}} + 2\delta \leq_{{\tiny \circled{1}}}\\\leq
	 \frac{A_k}{A_{k+1}}F(x_k) + \frac{\alpha_{k+1}}{A_{k+1}}(F(x) + \frac{1}{\alpha_{k+1}}V(x, u_k) - \frac{1}{\alpha_{k+1}}V(x, u_{k+1})) +\\+ \frac{3D\widetilde{\Omega}}{L_{k+1}m_{k+1}} + 2\delta +\frac{\alpha_{k+1}}{A_{k+1}}(\langle \widetilde{\nabla}^{m_{k+1}} f_\delta(y_{k+1})-\nabla f_\delta(y_{k+1}), x - y_{k+1} \rangle) +\\+ \frac{\alpha_{k+1}}{A_{k+1}}\langle \widetilde{\nabla}^{m_{k+1}} f_\delta(y_{k+1})-\nabla f_\delta(y_{k+1}), y_{k+1} - u_k \rangle
	 =\\=
	 	 \frac{A_k}{A_{k+1}}F(x_k) + \frac{\alpha_{k+1}}{A_{k+1}}(F(x) + \frac{1}{\alpha_{k+1}}V(x, u_k) - \frac{1}{\alpha_{k+1}}V(x, u_{k+1})) +\\+ \frac{3D\widetilde{\Omega}}{L_{k+1}m_{k+1}} + 2\delta +\frac{\alpha_{k+1}}{A_{k+1}}(\langle \widetilde{\nabla}^{m_{k+1}} f_\delta(y_{k+1})-\nabla f_\delta(y_{k+1}), x - u_{k} \rangle)
\end{gather*}

{\small \circled{1}} - из левой части \ref{exitLDLOrig2} и $A_{k}(y_{k+1} - x_k) = \alpha_{k+1} (u_k - y_{k+1})$ из (\ref{eqymir2DLST}).

\end{proof}

Нам будет полезен факт из \cite{lan2012validation}\cite{devolder2013exactness}
\begin{lemma}
\label{lemmaDev}
Пусть $\gamma_1$,...,$\gamma_k$ - i.i.d случайные величины, $\Gamma_k$ и $\nu_k$ - неслучайные функции от $\gamma_i$, и $c_i$ - неслучайные константы, для которых верно следующее
\begin{gather*}
\mathbb{E}(\Gamma_i|\gamma_1,\dots,\gamma_{i-1}) = 0\\
\abs{\Gamma_i} \leq c_i \nu_i\\
\mathbb{E}\Bigg(\exp\Bigg(\frac{\nu_i^2}{\sigma^2}\Bigg)\Bigg|\gamma_1,\dots,\gamma_{i-1}\Bigg) \leq \exp(1)
\end{gather*}

Тогда
\begin{gather*}
\mathbb{P}\Bigg(\sum_{i=1}^{k}\Gamma_i \geq \sqrt{3}\sqrt{\widehat{\Omega}}\sigma\sqrt{\sum_{i=1}^{k}c_i^2}\Bigg) \leq \exp(-\widehat{\Omega}) \,\,\forall k;\forall \widehat{\Omega} \geq 0
\end{gather*}

\end{lemma}

\begin{lemma}
	\label{mainTheoremDLST}
	Пусть $x_*$ - решения задачи (\ref{mainTask3}), тогда с вероятностью $1 - 2\beta$
	\begin{equation*}
	F(x_N) - F(x_*)\leq \frac{R^2}{A_{N}} + 2\delta N + \epsilon + D_Q\sqrt{\frac{\epsilon}{A_N}}
	\end{equation*}
\end{lemma}
\begin{proof}

	Учтем (\ref{eqymir2DLSTMK}) и (\ref{eqymir2DLSTA}) в Лемме \ref{lemma_maxmin_3DLST_2}, тогда с вероятностью большой или равной $1 - \beta$ $\forall k \geq 0$
	\begin{gather*}
		A_{k+1} F(x_{k+1}) - A_{k} F(x_{k}) + V(x, u_{k+1}) - V(x, u_{k}) \leq\\\leq \alpha_{k+1}F(x) + 2\delta A_{k+1} + \alpha_{k+1}\epsilon + \alpha_{k+1}\langle \widetilde{\nabla}^{m_{k+1}} f_\delta(y_{k+1})-\nabla f_\delta(y_{k+1}), x - u_k \rangle
	\end{gather*}
	
	Просуммируем неравенства по $k = 0, ..., N - 1$, 
	\begin{gather*}
		A_{N} F(x_N) - A_{0} F(x_0) + V(x, u_N) - V(x, u_0) \leq (A_N - A_0)F(x) +\\+ 2\delta\sum_{k = 0}^{N-1}A_{k+1} + \sum_{k = 0}^{N-1}\alpha_{k+1}\epsilon + \sum_{k = 0}^{N-1}\alpha_{k+1}\langle \widetilde{\nabla}^{m_{k+1}} f_\delta(y_{k+1})-\nabla f_\delta(y_{k+1}), x - u_k \rangle
	\end{gather*}
	Откуда, с учетом неравенства $V(x, u_N)\geq 0\,\,\,\forall x \in Q$
	\begin{gather*}
		A_{N} F(x_N) - A_NF(x)  \leq V(x, u_0) +\\+
		 2\delta\sum_{k = 0}^{N-1}A_{k+1}+A_N\epsilon+ \sum_{k = 0}^{N-1}\alpha_{k+1}\langle \widetilde{\nabla}^{m_{k+1}} f_\delta(y_{k+1})-\nabla f_\delta(y_{k+1}), x - u_k \rangle
	\end{gather*}
	
	Возьмем $x = x_*$, оценим $A_{k+1}$ через $A_N$.
	\begin{gather*}
	A_{N} F(x_N) - A_NF(x_*)  \leq V(x_*, u_0) +\\+
			 2\delta N A_{N}+A_N\epsilon+ \sum_{k = 0}^{N-1}\alpha_{k+1}\langle \widetilde{\nabla}^{m_{k+1}} f_\delta(y_{k+1})-\nabla f_\delta(y_{k+1}), x - u_k \rangle\\
	F(x_N) - F(x_*)\leq \frac{R^2}{A_{N}} + 2\delta N + \epsilon + \sum_{k = 0}^{N-1}\frac{\alpha_{k+1}}{A_N}\langle \widetilde{\nabla}^{m_{k+1}} f_\delta(y_{k+1})-\nabla f_\delta(y_{k+1}), x - u_k \rangle
	\end{gather*}
	
	Воспользуемся Леммой \ref{lemmaDev} для последнего слагаемого в неравенстве c $\gamma_i =\xi_i$, $\sigma^2 = D$, $\Gamma_i = \frac{\alpha_{k_i+1}}{A_N m_{k_i+1}}\langle\nabla f_\delta(y_{k_i};\xi_i) - \nabla f_\delta(y_{k_i}), x - u_{(k_i - 1)}\rangle$, $c_{i} = \frac{D_Q\alpha_{k_i+1}}{A_N m_{k_i+1}}$,$\nu_i = \norm{\nabla f_\delta(y_{k_i};\xi_i) - \nabla f_\delta(y_{k_i})}_*$ с $i \in [1,\dots,\sum_{k=0}^{N-1}m_{k+1}]$, где $k_i$ равно $k + 1$ для всех $i \in [m_{k} + 1,\dots, m_{k+1}]$. Выберем $\widehat{\Omega} = \ln(\frac{1}{\beta}) \leq \widetilde{\Omega}$, тогда с вероятностью не меньше $1 - 2\beta$
	\begin{gather*}
	F(x_N) - F(x_*)\leq \frac{R^2}{A_{N}} + 2\delta N + \epsilon + \sqrt{3}\sqrt{\widehat{\Omega}D}\sqrt{\sum_{i=0}^{N-1}\frac{D_Q^2\alpha^2_{k+1}}{A_N^2m_{k+1}}}
	\end{gather*}
	
	Учтем (\ref{eqymir2DLSTMK})
	\begin{gather*}
	F(x_N) - F(x_*)\leq \frac{R^2}{A_{N}} + 2\delta N + \epsilon + D_Q\frac{\sqrt{\widehat{\Omega}}}{\sqrt{\widetilde{\Omega}}}\sqrt{\sum_{i=0}^{N-1}\frac{\alpha_{k+1}\epsilon}{A_N^2}}\\
	F(x_N) - F(x_*)\leq \frac{R^2}{A_{N}} + 2\delta N + \epsilon + D_Q\sqrt{\frac{\epsilon}{A_N}}
	\end{gather*}
	
\end{proof}

\begin{theorem}
\label{mainTheoremDLST2}
Пусть $\delta \leq \frac{\epsilon^\frac{3}{2}}{6\sqrt{3}\sqrt{L}D_Q}$. Тогда с вероятностью $1 - 3\beta$
\begin{gather*}
F(x_N) - F(x_*)\leq 4\epsilon.
	\end{gather*}
\end{theorem}

\begin{proof}
Мы знаем из Леммы \ref{lemma_maxmin_DLST1}, что с вероятностью $1 - \beta$ верно неравенство $A_N \geq \frac{(N+1)^2}{12L}$, отсюда и условия на $N$, получаем, что

\begin{gather*}
\frac{R^2}{A_{N}} \leq \frac{D_Q^2}{A_{N}} \leq \frac{12LD_Q^2}{\left(\frac{2\sqrt{3}\sqrt{L}D_Q}{\sqrt{\epsilon}}+1\right)^2}\leq
\frac{D_Q^2\epsilon}{D_Q^2}=\epsilon
\end{gather*}

и

\begin{gather*}
D_Q\sqrt{\frac{\epsilon}{A_N}} \leq D_Q\sqrt{\frac{\epsilon^2}{D_Q^2}} =\epsilon
\end{gather*}.

Помимо этого с вероятностью $1 - \beta$ из Леммы \ref{mainTheoremDLST}

\begin{equation*}
F(x_N) - F(x_*)\leq \frac{R^2}{A_{N}} + 2\delta N + \epsilon + D_Q\sqrt{\frac{\epsilon}{A_N}}
\end{equation*}

Объединяя все вместе, включая условие на $\delta$, получаем, что с вероятностью $1 - 3\beta$

\begin{equation*}
F(x_N) - F(x_*) \leq 4\epsilon
\end{equation*}

\end{proof}

\begin{remark}
\leavevmode
Далее все утверждения будут выполнятся с вероятностью не меньше $1 - 3\beta$. Оценим количество обращений $M$ к оракулу за стохастическими градиентами. По ходу алгоритма мы можем контролировать $\frac{D_Q^2}{A_N}$, тогда пусть $\widetilde{N} + 1$ - минимальное число шагов, для которого выполнено $\frac{D_Q^2}{A_{\widetilde{N} + 1}} \leq \epsilon$. Ясно, что $\widetilde{N} + 1 \leq N$, причем условие $\frac{D_Q^2}{A_{\widetilde{N} + 1}} \leq \epsilon$ является достаточным условием для достижения $\epsilon$-решения по функции. Как и в Замечании \ref{remark_maxmin} количество обращений за $\widetilde{\nabla}^{m_{k+1}} f_\delta(y_{k+1})$ на $k$-ом шаге будет равно $2 + \log\left(\frac{L_k}{L_{k-1}}\right)$. $L_{\widetilde{N} + 1} \leq 3L$. Поэтому общее количество обращений к оракулу за $\nabla f_\delta(y;\xi)$ будет равно

\begin{gather*}
M = 2\sum_{k=1}^{\widetilde{N} + 1}m_{k}\left(2 + \log\left(\frac{L_k}{L_{k-1}}\right)\right) = 4\sum_{k=1}^{\widetilde{N} + 1}m_{k} + 2\sum_{k=1}^{\widetilde{N} + 1}m_{k}\log\left(\frac{L_k}{L_{k-1}}\right) \leq\\\leq
4\sum_{k=1}^{\widetilde{N} + 1}m_{k} + 2\sum_{j=1}^{\widetilde{N} + 1}m_{j}\sum_{k=1}^{\widetilde{N} + 1}\log\left(\frac{L_k}{L_{k-1}}\right) =\\
\left(4 + \log\left(\frac{L_{\widetilde{N} + 1}}{L_0}\right)\right)\sum_{k=1}^{\widetilde{N} + 1}m_{k} \leq \left(4 + \log\left(\frac{3L}{L_0}\right)\right)\sum_{k=1}^{\widetilde{N} + 1}m_{k}
\end{gather*}

Рассмотрим $\sum_{k=1}^{\widetilde{N} + 1}m_{k}$:

\begin{gather*}
\sum_{k=0}^{\widetilde{N}}m_{k+1}\leq {\widetilde{N} + 1} + \frac{3D\widetilde{\Omega}}{\epsilon}A_{\widetilde{N} + 1}
\end{gather*}

Так как
\begin{gather*}
\alpha_{\widetilde{N}+1} = \frac{1}{2L_{\widetilde{N}+1}} + \sqrt{\frac{1}{4L_{\widetilde{N}+1}^2} + \frac{A_{\widetilde{N}}}{L_{\widetilde{N}+1}}} \leq_{{\tiny \circled{1}}} \frac{1}{L_{\widetilde{N}+1}} + \sqrt{\frac{A_{\widetilde{N}}}{L_{\widetilde{N}+1}}} \leq \frac{2}{L_{\widetilde{N}}} + \sqrt{\frac{2A_{\widetilde{N}}}{L_{\widetilde{N}}}} =\\= \frac{2}{L_{\widetilde{N}}} + \sqrt{2}\alpha_{\widetilde{N}}
\end{gather*}

{\small \circled{1}} - из $\sqrt{x + y} \leq \sqrt{x} + \sqrt{y}$.

и
\begin{gather*}
\alpha_{\widetilde{N}} = \frac{1}{2L_{\widetilde{N}}} + \sqrt{\frac{1}{4L_{\widetilde{N}}^2} + \frac{A_{\widetilde{N} - 1}}{L_{\widetilde{N}}}} \geq \frac{1}{2L_{\widetilde{N}}}
\end{gather*}

Отсюда

\begin{gather*}
\alpha_{\widetilde{N} + 1} \leq (4 + \sqrt{2})\alpha_{\widetilde{N}} \leq (4 + \sqrt{2})A_{\widetilde{N}} \leq 6A_{\widetilde{N}}
\end{gather*}

Так как $\frac{D_Q^2}{A_{\widetilde{N}}} > \epsilon$, это следует из минимальности $\widetilde{N} + 1$, то $\frac{D_Q^2}{A_{\widetilde{N} + 1}} = \frac{D_Q^2}{A_{\widetilde{N}} + \alpha_{\widetilde{N} + 1}} \geq \frac{D_Q^2}{A_{\widetilde{N}} + 6A_{\widetilde{N}}} > \frac{\epsilon}{7}$

и
\begin{gather*}
\sum_{k=0}^{\widetilde{N}}m_{k+1} \leq {\widetilde{N} + 1} + \frac{3D\widetilde{\Omega}}{\epsilon}A_{\widetilde{N} + 1} \leq {\widetilde{N} + 1} + \frac{21D\widetilde{\Omega}D_Q^2}{\epsilon^2}.
\end{gather*}

В конечном счете

\begin{gather*}
M  \leq \left(4 + \log\left(\frac{3L}{L_0}\right)\right)\left({\widetilde{N} + 1} + \frac{21D\widetilde{\Omega}D_Q^2}{\epsilon^2}\right) \leq\\\leq
\left(4 + \log\left(\frac{3L}{L_0}\right)\right)\left(\frac{2\sqrt{3}\sqrt{L}D_Q}{\sqrt{\epsilon}} + \frac{21D\widetilde{\Omega}D_Q^2}{\epsilon^2} + 1\right).
\end{gather*}

\end{remark}

\section{Спуск по направлению с неточным оракулом} \label{sec:rand}

\hspace{0.6cm}Пусть $f(x)$ удовлетворяет всем условиям из раздела. \ref{sec:mmt}. Определим следующие величины
\begin{definition}
\leavevmode
\label{defRand1}
\begin{enumerate}
\item
$e_{k+1}$ - случайный вектор на евклидовой сфере радиуса 1, который удовлетворяет следующим условиям: $\mathbb{E} e_{k+1}^i e_{k+1}^j = 0 \,\,\forall i,j$ и $\mathbb{E} (e_{k+1}^i)^2 = \frac{1}{n}\,\,\forall i$, где $e_{k+1}^i$ - $i$ компонента случайного вектора $e_{k+1}$.
\item
$\widetilde{\delta}_{k+1} \in \mathds{R}$ - случайный шум, про который только известно, что он ограничен: $\abs{\widetilde{\delta}_{k+1}} \leq \delta$.
\item
$\norm{x}^2_L \stackrel{def}{=} L\sum_{i = 1}^{n}x_i^2$.
\item
$\widetilde{\nabla}f(y) \stackrel{def}{=} n(\langle\nabla f(y), e_{k+1}\rangle + \widetilde{\delta}_{k+1}) e_{k+1}$ - это аппроксимация производной по направлению $e_{k+1}$.
\end{enumerate}
\end{definition}

 Будем предполагать, что оракул вместо $\nabla f(y)$ выдает $\widetilde{\nabla}f(y)$. Далее вместо произвольной нормы будем использовать $\norm{}_L$ и будем брать $d(x) = \frac{1}{2}\norm{x}_L^2$. Как следствие можно заметить, что $V(x, y) = d(x) - d(y) - \langle\nabla d(y), x - y\rangle = \frac{1}{2}\norm{x - y}^2_L$. Дополнительно будем считать, что $Q = \mathds{R}^n$.

Рассматривается следующая задача оптимизации
\begin{align}
\label{mainTask2}
f(x) \rightarrow \min_{x \in \mathds{R}^n}
\end{align}

Опишем алгоритм зеркального метода треугольника для оракула с производной по направлению.

\begin{mdframed}
\textbf{0 - шаг:}
\begin{gather}
x_0 = u_0 = y_0\\
\alpha_0 = 1 - \frac{1}{n}\\
A_0 = \alpha_0
\end{gather}
\textbf{$\boldsymbol{k+1}$ - шаг:}
\begin{gather}
y_{k+1} = \frac{\alpha_{k+1}u_k + A_k x_k}{A_{k+1}} \label{eqymir3}
\end{gather}

\begin{center}
Сгенерировать: $\widetilde{\nabla} f(y_{k+1})$.
\end{center}
\begin{gather}
\alpha_{k+1} = \frac{k + 2n}{2n^2}\\
A_{k+1} = A_k + \alpha_{k+1}\\
\phi_{k+1}(x) = V(x, u_k) + \alpha_{k+1}\langle\widetilde{\nabla} f(y_{k+1}), x \rangle\\
u_{k+1} = \argmin_{x \in \mathds{R}^n}\phi_{k+1}(x) \label{equmir3}\\
x_{k+1} = y_{k+1} + n\frac{\alpha_{k+1}}{A_{k+1}} (u_{k+1} - u_k) \label{eqxmir3}
\end{gather}

\end{mdframed}

\begin{lemma}
	\label{lemAa2_3}
	Пусть для последовательности $\alpha_k$ выполнено
	\begin{align*}
	\alpha_0 = 1 - \frac{1}{n}\\
	A_k = \sum_{i = 0}^{k}\alpha_i\\
	\alpha_i = \frac{i - 1 + 2n}{2n^2}
	\end{align*}
	Тогда верны следующие неравенства $\forall k \geq 1$

	\begin{gather*}
	A_k = \frac{(k - 1 + 2n)^2 + k - 1}{4n^2}\\
	A_k \geq n^2\alpha_k^2 = \frac{(k - 1 + 2n)^2}{4n^2}\\
	A_k \leq \frac{(k - 1 + 2n)^2}{2n^2}
	\end{gather*}
	 
\end{lemma}

При доказательстве основной теоремы нам потребуется следующая лемма, которая была доказана по аналогии с \cite{fercoq2015accelerated}.

\begin{lemma}
	\label{lemAa3}
	$\forall k\geq 0$ $x_{k+1}$, $y_{k+1}$ есть выпуклая комбинация $u_0 \dots u_{k+1}$. Причем $x_{k+1} = \sum_{l=0}^{k+1}\gamma_{k+1}^l u_{l}$, где $\gamma_0^0 = 1$, $\gamma_1^0 = 0$, $\gamma_0^1 = 1$ и для $k \geq 1$,
	\begin{gather*}
	\gamma_{k+1}^l =\begin{cases} (1 - \frac{\alpha_{k+1}}{A_{k+1}})\gamma_k^l, & l = 0,\dots,k - 1 \\ \frac{\alpha_{k+1}}{A_{k+1}} (1 - n \frac{\alpha_{k}}{A_{k}}) + n (\frac{\alpha_{k}}{A_{k}} - \frac{\alpha_{k+1}}{A_{k+1}}),& l = k \\
	n \frac{\alpha_{k+1}}{A_{k+1}},& l = k+1 \end{cases}
	\end{gather*}
\end{lemma}

\begin{proof}

	Сначала отметим, что если $x_{k}$ есть выпуклая комбинация $u_0 \dots u_{k}$, то и $y_{k+1}$ - выпуклая комбинация $u_0 \dots u_{k}$, это следует из (\ref{eqymir3}).
	
	Докажем для $x_{k+1}$ теперь. Так как $x_0 = u_0$, то $\gamma_0^0 = 1$. Рассмотрим для $k = 0$. $x_1 = y_{1} + n\frac{\alpha_{1}}{A_{1}} (u_{1} - u_0) = u_0 + n\frac{\alpha_{1}}{A_{1}} (u_{1} - u_0) = (1 - n\frac{\alpha_{1}}{A_{1}}) u_0 + n\frac{\alpha_{1}}{A_{1}} u_{1}$. Так как $n\frac{\alpha_{1}}{A_{1}} = 1$, $x_1 = u_{1}$, $\gamma_1^0 = 0$ и $\gamma_1^0 = 1$.
	
	Пусть данное утверждение верно для $k$, докажем для $k+1$.
	
	\begin{gather*}
	x_{k+1} = y_{k+1} + n\frac{\alpha_{k+1}}{A_{k+1}} (u_{k+1} - u_k) = \frac{\alpha_{k+1}u_k + A_k x_k}{A_{k+1}} + n\frac{\alpha_{k+1}}{A_{k+1}} (u_{k+1} - u_k) = \\ =
	\frac{A_k}{A_{k+1}}x_k + (\frac{\alpha_{k+1}}{A_{k+1}} - n\frac{\alpha_{k+1}}{A_{k+1}})u_k + n\frac{\alpha_{k+1}}{A_{k+1}}u_{k+1}
	\end{gather*}
	
	Легко заметить, что коэффициенты при векторах в сумме дают $1$. Теперь распишем $x_k$, 
	
	\begin{gather*}
	x_{k+1} = 
	\frac{A_k}{A_{k+1}}x_k + (\frac{\alpha_{k+1}}{A_{k+1}} - n\frac{\alpha_{k+1}}{A_{k+1}})u_k + n\frac{\alpha_{k+1}}{A_{k+1}}u_{k+1} = \\ =
	(1 - \frac{\alpha_{k+1}}{A_{k+1}}) \sum_{l = 0}^{k} \gamma_k^l u_l + (\frac{\alpha_{k+1}}{A_{k+1}} - n\frac{\alpha_{k+1}}{A_{k+1}})u_k + n\frac{\alpha_{k+1}}{A_{k+1}}u_{k+1} = \\ =
	(1 - \frac{\alpha_{k+1}}{A_{k+1}}) \sum_{l = 0}^{k-1} \gamma_k^l u_l + \Big(\gamma_k^k (1 - \frac{\alpha_{k+1}}{A_{k+1}}) + (\frac{\alpha_{k+1}}{A_{k+1}} - n\frac{\alpha_{k+1}}{A_{k+1}})\Big)u_k + n\frac{\alpha_{k+1}}{A_{k+1}}u_{k+1} = \\=
	(1 - \frac{\alpha_{k+1}}{A_{k+1}}) \sum_{l = 0}^{k-1} \gamma_k^l u_l + \Big(n \frac{\alpha_{k}}{A_{k}} (1 - \frac{\alpha_{k+1}}{A_{k+1}}) + (\frac{\alpha_{k+1}}{A_{k+1}} - n\frac{\alpha_{k+1}}{A_{k+1}})\Big)u_k + n\frac{\alpha_{k+1}}{A_{k+1}}u_{k+1} = \\=
	(1 - \frac{\alpha_{k+1}}{A_{k+1}}) \sum_{l = 0}^{k-1} \gamma_k^l u_l + \Big(\frac{\alpha_{k+1}}{A_{k+1}} (1 - n\frac{\alpha_{k}}{A_{k}}) + n(\frac{\alpha_{k}}{A_{k}} -\frac{\alpha_{k+1}}{A_{k+1}})\Big)u_k + n\frac{\alpha_{k+1}}{A_{k+1}}u_{k+1}
	\end{gather*}
	
	Осталось показать, что $\gamma_{k+1}^l \geq 0$, $l = 0,\dots,k+1$.
	
	Легко увидеть, что это верно для $\gamma_{k+1}^l$, $l = 0,\dots,k-1$ и $\gamma_{k+1}^{k+1}$. $\gamma_{k+1}^{k} = \frac{\alpha_{k+1}}{A_{k+1}} (1 - n\frac{\alpha_{k}}{A_{k}}) + n(\frac{\alpha_{k}}{A_{k}} -\frac{\alpha_{k+1}}{A_{k+1}})$, так как $\frac{\alpha_{k}}{A_{k}} \geq \frac{\alpha_{k+1}}{A_{k+1}}$ и $\frac{\alpha_{k}}{A_{k}} \leq \frac{1}{n}$ для $k \geq 1$, получаем $\gamma_{k+1}^{k} \geq 0$. 
	
\end{proof}

\begin{lemma}
	$\forall u \in \mathds{R}^n$ выполнено
	\begin{gather*}
	\alpha_{k+1}\langle \widetilde{\nabla} f(y_{k+1}), u_k - u\rangle \leq A_{k+1}(f(y_{k+1}) - f(x_{k+1})) + V(u, u_k) - V(u, u_{k+1}) +\\
	+ \frac{A_{k+1}}{L}\langle\nabla f(y_{k+1}), \widetilde{\delta}_{k+1} e_{k+1}\rangle + \frac{A_{k+1}}{L} \widetilde{\delta}_{k+1}^2
	\end{gather*}
	\label{lemAZ1_3}
\end{lemma}

\begin{proof}
	\begin{gather*}
	\alpha_{k+1}\langle \widetilde{\nabla} f(y_{k+1}), u_k - u\rangle = \\=
	\alpha_{k+1}\langle \widetilde{\nabla} f(y_{k+1}), u_k - u_{k+1}\rangle + \alpha_{k+1}\langle \widetilde{\nabla} f(y_{k+1}), u_{k+1} - u\rangle \leq_{{\tiny \circled{1}}} \\\leq
	\alpha_{k+1}\langle \widetilde{\nabla} f(y_{k+1}), u_k - u_{k+1}\rangle + 
	\langle -\nabla V(u_{k+1}, u_k), u_{k+1} - u \rangle = \\=
	 \alpha_{k+1}\langle \widetilde{\nabla} f(y_{k+1}), u_k - u_{k+1}\rangle + V(u, u_k) - V(u, u_{k+1}) - V(u_{k+1}, u_k) \leq \\\leq
	 \alpha_{k+1}\langle \widetilde{\nabla} f(y_{k+1}), u_k - u_{k+1}\rangle + V(u, u_k) - V(u, u_{k+1}) - \frac{1}{2}\norm{u_k - u_{k+1}}^2_L \leq_{{\tiny \circled{2}}} \\\leq
	 A_{k+1}\langle \nabla f(y_{k+1}), y_{k+1} - x_{k+1}\rangle + V(u, u_k) - V(u, u_{k+1}) -\\- \frac{A_{k+1}^2}{2n^2\alpha_{k+1}^2}\norm{y_{k+1} - x_{k+1}}^2_L+ \frac{A_{k+1}}{L}\langle\nabla f(y_{k+1}), \widetilde{\delta}_{k+1} e_{k+1}\rangle + \frac{A_{k+1}}{L} \widetilde{\delta}_{k+1}^2 \leq \\\leq
	 A_{k+1}(\langle \nabla f(y_{k+1}), y_{k+1} - x_{k+1}\rangle - \frac{1}{2}\norm{y_{k+1} - x_{k+1}}^2_L) + V(u, u_k) - V(u, u_{k+1}) +\\+ \frac{A_{k+1}}{L}\langle\nabla f(y_{k+1}), \widetilde{\delta}_{k+1} e_{k+1}\rangle + \frac{A_{k+1}}{L} \widetilde{\delta}_{k+1}^2 \leq_{{\tiny \circled{3}}} \\\leq
	 A_{k+1}(f(y_{k+1}) - f(x_{k+1})) + V(u, u_k) - V(u, u_{k+1}) +\\+ \frac{A_{k+1}}{L}\langle\nabla f(y_{k+1}), \widetilde{\delta}_{k+1} e_{k+1}\rangle + \frac{A_{k+1}}{L} \widetilde{\delta}_{k+1}^2
	\end{gather*}

{\small \circled{1}} - из (\ref{equmir3})

{\small \circled{2}} - из:

Так как $\phi_{k+1}(x)$ сильно выпуклая и оптимизация происходит на $\mathds{R}^n$, то
\begin{gather*}
	\nabla \phi_{k+1}(u_{k+1}) = 0 \\
	u_{k+1} = u_{k} - \frac{\alpha_{k+1}}{L} \widetilde{\nabla}f(y_{k+1})
\end{gather*}

и

\begin{gather*}
	y_{k+1} - x_{k+1} = n\frac{\alpha_{k+1}}{A_{k+1}} (u_{k} - u_{k+1}) = \\=
	n\frac{\alpha_{k+1}^2}{A_{k+1}L} \widetilde{\nabla}f(y_{k+1}) =\\=
	n^2\frac{\alpha_{k+1}^2}{A_{k+1}L} (\langle\nabla f(y_{k+1}), e_{k+1}\rangle  + \widetilde{\delta}_{k+1})e_{k+1}
\end{gather*}

Отсюда

\begin{gather*}
	\alpha_{k+1}\langle \widetilde{\nabla} f(y_{k+1}), u_k - u_{k+1}\rangle = \\=
	\frac{\alpha_{k+1}}{L}\langle n (\langle\nabla f(y_{k+1}), e_{k+1}\rangle + \widetilde{\delta}_{k+1}) e_{k+1}, \alpha_{k+1} n (\langle\nabla f(y_{k+1}), e_{k+1}\rangle  + \widetilde{\delta}_{k+1})e_{k+1} \rangle = \\=
	\frac{\alpha_{k+1}^2 n^2}{L}(\langle\nabla f(y_{k+1}), e_{k+1}\rangle + \widetilde{\delta}_{k+1})^2 \langle e_{k+1}, e_{k+1}\rangle = \\=
	\frac{\alpha_{k+1}^2 n^2}{L}\langle\nabla f(y_{k+1}), e_{k+1}\rangle(\langle\nabla f(y_{k+1}), e_{k+1}\rangle + \widetilde{\delta}_{k+1}) +\\+ \frac{\alpha_{k+1}^2 n^2}{L}\widetilde{\delta}_{k+1}(\langle\nabla f(y_{k+1}), e_{k+1}\rangle + \widetilde{\delta}_{k+1})= \\=
	\frac{\alpha_{k+1}^2 n^2}{L}\langle\nabla f(y_{k+1}), (\langle\nabla f(y_{k+1}), e_{k+1}\rangle + \widetilde{\delta}_{k+1})e_{k+1}\rangle +\\+ \frac{\alpha_{k+1}^2 n^2}{L}\widetilde{\delta}_{k+1}(\langle\nabla f(y_{k+1}), e_{k+1}\rangle + \widetilde{\delta}_{k+1})= \\=
	A_{k+1}\langle\nabla f(y_{k+1}), y_{k+1} - x_{k+1}\rangle + \frac{\alpha_{k+1}^2 n^2}{L}\langle\nabla f(y_{k+1}), \widetilde{\delta}_{k+1} e_{k+1}\rangle + \frac{\alpha_{k+1}^2 n^2}{L} \widetilde{\delta}_{k+1}^2 \leq \\ \leq
	A_{k+1}\langle\nabla f(y_{k+1}), y_{k+1} - x_{k+1}\rangle + \frac{A_{k+1}}{L}\langle\nabla f(y_{k+1}), \widetilde{\delta}_{k+1} e_{k+1}\rangle + \frac{A_{k+1}}{L} \widetilde{\delta}_{k+1}^2
\end{gather*}

{\small \circled{3}} - из Липщевости
\end{proof}
Пусть $\mathbb{E}_{k}$ - условное математическое ожидание по $k$ итерации относительно $1, ..., k-1$ итерации. $R_k = \norm{u_k - u}_L$, $M_k = \norm{\nabla f(y_{k+1})}_2$.

\begin{corollary}
$\forall u \in \mathds{R}^n$ выполнено
	\begin{gather*}
	\alpha_{k+1}\langle \nabla f(y_{k+1}), u_k - u\rangle \leq A_{k+1}(f(y_{k+1}) - \mathbb{E}_{k+1} f(x_{k+1})) + V(u, u_k) -\\- \mathbb{E}_{k+1} V(u, u_{k+1}) + \frac{A_{k+1}}{L}\delta^2 +
		A_{k+1}\delta \frac{M_k}{L\sqrt{n}} + \alpha_{k+1} \delta \frac{\sqrt{n}}{\sqrt{L}} R_k
	\end{gather*}
	\label{corMailLemma}
\end{corollary}

\begin{proof}
	Возьмем $\mathbb{E}_{k+1}$ от обеих частей неравенства леммы \ref{lemAZ1_3}
	\begin{gather*}
	\mathbb{E}_{k+1}\alpha_{k+1}\langle \widetilde{\nabla} f(y_{k+1}), u_k - u\rangle \leq \mathbb{E}_{k+1}A_{k+1}(f(y_{k+1}) - f(x_{k+1})) +\\+ \mathbb{E}_{k+1}V(u, u_k) - \mathbb{E}_{k+1}V(u, u_{k+1}) + \frac{A_{k+1}}{L}\mathbb{E}_{k+1}\langle\nabla f(y_{k+1}), \widetilde{\delta}_{k+1} e_{k+1}\rangle + \frac{A_{k+1}}{L} \mathbb{E}_{k+1}\widetilde{\delta}_{k+1}^2
	\end{gather*}
	
	Воспользуемся тем, что 
	\begin{gather*}
	\mathbb{E}_{k+1} \langle n\langle\nabla f(y_{k+1}), e_{k+1}\rangle e_{k+1}, u_k - u\rangle = n\langle \mathbb{E}_{k+1}\langle\nabla f(y_{k+1}), e_{k+1}\rangle e_{k+1}, u_k - u\rangle =_{{\tiny \circled{1}}}\\= \langle \nabla f(y_{k+1}), u_k - u\rangle
	\end{gather*}
	
	{\small \circled{1}} - из определения \ref{defRand1}.1
	
	Тогда
	\begin{gather*}
	\alpha_{k+1}\langle \nabla f(y_{k+1}), u_k - u\rangle + \mathbb{E}_{k+1}\alpha_{k+1}\langle n\widetilde{\delta}_{k+1} e_{k+1} , u_k - u\rangle \leq\\\leq
	A_{k+1}(f(y_{k+1}) - \mathbb{E}_{k+1}f(x_{k+1})) + V(u, u_k) - \mathbb{E}_{k+1}V(u, u_{k+1}) +\\+ \frac{A_{k+1}}{L}\mathbb{E}_{k+1}\langle\nabla f(y_{k+1}), \widetilde{\delta}_{k+1} e_{k+1}\rangle + \frac{A_{k+1}}{L} \mathbb{E}_{k+1}\widetilde{\delta}_{k+1}^2
	\end{gather*}
	
	Используя условия из определения \ref{defRand1} на $e_{k+1}$ и $\widetilde{\delta}_{k+1}$, можно показать, что
	\begin{gather*}
	\mathbb{E}_{k+1}\langle\nabla f(y_{k+1}), \widetilde{\delta}_{k+1} e_{k+1}\rangle \leq \delta\mathbb{E}_{k+1}|\langle\nabla f(y_{k+1}),  e_{k+1}\rangle| = \delta\mathbb{E}_{k+1}\sqrt{\langle\nabla f(y_{k+1}),  e_{k+1}\rangle^2} \leq \\ \leq
	\delta \sqrt{\mathbb{E}_{k+1}\langle\nabla f(y_{k+1}), e_{k+1}\rangle^2} = 
	\delta \sqrt{(\nabla f(y_{k+1}))^T \mathbb{E}_{k+1} e_{k+1} e_{k+1}^T \nabla f(y_{k+1})} = \\ =
	\frac{\delta}{\sqrt{n}} \norm{\nabla f(y_{k+1})}_2 = \delta\frac{M_k}{\sqrt{n}}
	\end{gather*}
	
	Аналогично
	\begin{gather*}
	\mathbb{E}_{k+1}\langle n\widetilde{\delta}_{k+1} e_{k+1} , u_k - u\rangle \geq
	-\delta n \mathbb{E}_{k+1}|\langle e_{k+1} , u_k - u\rangle| \geq \\ \geq
	-\delta n \sqrt{\mathbb{E}_{k+1} \langle e_{k+1} , u_k - u\rangle^2} =
	-\delta \sqrt{n} \norm{u_k - u}_2 = -\frac{\delta \sqrt{n}}{\sqrt{L}} R_k
	\end{gather*} 
	В конечном счете получим, что
	\begin{gather*}
	\alpha_{k+1}\langle \nabla f(y_{k+1}), u_k - u\rangle \leq
	A_{k+1}(f(y_{k+1}) - \mathbb{E}_{k+1}f(x_{k+1})) +\\+ V(u, u_k) - \mathbb{E}_{k+1}V(u, u_{k+1}) + \frac{A_{k+1}}{L}\delta^2 +
	A_{k+1}\delta \frac{M_k}{L\sqrt{n}} + \alpha_{k+1} \frac{\delta \sqrt{n}}{\sqrt{L}} R_k
	\end{gather*}
\end{proof}

\begin{lemma}
	$\forall u \in \mathds{R}^n$ выполнено
	\begin{gather*}
		A_{k+1} \mathbb{E}_{k+1}f(x_{k+1}) - A_{k} f(x_{k}) + \mathbb{E}_{k+1}V(u, u_{k+1}) - V(u, u_{k}) \leq \alpha_{k+1}f(u)+\\+ \frac{A_{k+1}}{L}\delta^2 +
				A_{k+1}\delta \frac{M_k}{L\sqrt{n}} + \alpha_{k+1} \frac{\delta \sqrt{n}}{\sqrt{L}} R_k
	\end{gather*}
	\label{lemAZ2_3}
\end{lemma}

\begin{proof}
	\begin{gather*}
		\alpha_{k+1}(f(y_{k+1}) - f(u)) \leq \\ \leq
		\alpha_{k+1}\langle \nabla f(y_{k+1}), y_{k+1} - u\rangle = \\ =
		\alpha_{k+1}\langle \nabla f(y_{k+1}), y_{k+1} - u_k\rangle + 
		\alpha_{k+1}\langle \nabla f(y_{k+1}), u_k - u\rangle =_{{\tiny \circled{1}}} \\ =
		A_{k}\langle \nabla f(y_{k+1}), x_{k} - y_{k+1}\rangle + 
		\alpha_{k+1}\langle \nabla f(y_{k+1}), u_k - u\rangle \leq \\ \leq
		A_{k}(f(x_k) - f(y_{k+1})) + \alpha_{k+1}\langle \nabla f(y_{k+1}), u_k - u\rangle \leq_{{\tiny \circled{2}}} \\ \leq
		A_{k}(f(x_k) - f(y_{k+1})) + A_{k+1}(f(y_{k+1}) - \mathbb{E}_{k+1}f(x_{k+1})) +\\+ V(u, u_k) - \mathbb{E}_{k+1}V(u, u_{k+1}) + \frac{A_{k+1}}{L}\delta^2 +
			A_{k+1}\delta \frac{M_k}{L\sqrt{n}} + \alpha_{k+1} \frac{\delta \sqrt{n}}{\sqrt{L}} R_k
	\end{gather*}
	
	{\small \circled{1}} - из (\ref{eqymir3})
	
	{\small \circled{2}} - из Следствия \ref{corMailLemma}
	
	То есть 
	\begin{gather*}
		\alpha_{k+1}(f(y_{k+1}) - f(u)) \leq \\ \leq
		A_{k}(f(x_k) - f(y_{k+1})) + A_{k+1}(f(y_{k+1}) - \mathbb{E}_{k+1}f(x_{k+1})) +\\+ V(u, u_k) - \mathbb{E}_{k+1}V(u, u_{k+1}) + \frac{A_{k+1}}{L}\delta^2 +
			A_{k+1}\delta \frac{M_k}{L\sqrt{n}} + \alpha_{k+1} \frac{\delta \sqrt{n}}{\sqrt{L}} R_k
	\end{gather*}
	
	Отсюда получаем утверждение леммы.
	
\end{proof}

\begin{definition}
$\frac{1}{2}P_0^2 = \frac{1}{2}R_0^2 + (1 - \frac{1}{n}) (f(x_0) - f_*)$
\end{definition}

\begin{theorem}
Пусть $\epsilon$ фиксирована и выполнено
	\label{mainTheorem3}
	\begin{gather*}
	N = \ceil{\frac{\sqrt{2}nP_0}{\sqrt{\epsilon}} + 1 - 2n}\\
	\delta \leq \min\Bigg\{\frac{\epsilon^{\frac{3}{4}}\sqrt{L}}{4\sqrt[4]{2}\sqrt{nP_0}},
	 \frac{\epsilon^{\frac{3}{2}}\sqrt{L}}{96\sqrt{n}P_0^2}\Bigg\},
	\end{gather*} тогда
	\begin{gather*}
	\mathbb{E}f(x_N) - f(x_*) \leq 3\epsilon
	\end{gather*}
\end{theorem}

\begin{remark}
\leavevmode

Чтобы теорема была корректна, требуется, чтобы $N = \ceil{\frac{\sqrt{2}nP_0}{\sqrt{\epsilon}} + 1 - 2n} \geq 1$. Рассмотрим случай, когда $\ceil{\frac{\sqrt{2}nP_0}{\sqrt{\epsilon}} + 1 - 2n} \leq 0$, это эквивалентно
\begin{gather*}
\frac{\sqrt{2}nP_0}{\sqrt{\epsilon}} + 1 - 2n \leq 0
\end{gather*}
Далее
\begin{gather*}
\sqrt{\epsilon} \geq \frac{\sqrt{2}nP_0}{2n - 1} \geq \frac{\sqrt{2}P_0}{2}\\
\epsilon \geq \frac{1}{2}P_0^2
\end{gather*}
Выпишем условие Липшица для $f(x)$
\begin{gather*}
f(x_0) - f(x_*) \leq \langle\nabla f(x_*), x_0 - x_*\rangle + \frac{1}{2}\norm{x_0 - x_*}_L^2 \leq \frac{1}{2}R_0^2 \leq \frac{1}{2}P_0^2 \leq \epsilon
\end{gather*}
Если $N \leq 0$, то $x_0$ является $\epsilon$-решением задачи оптимизации, поэтому далее будем считать, что $N \geq 1$.
\end{remark}

Для начала докажем следующее вспомогательное утверждение.

\begin{lemma}
	\begin{gather*}
	\frac{1}{2}\mathbb{E} R_K^2 \leq P_0^2 \,\,\,\, \forall K \leq N
	\end{gather*}
	\label{lemAZ3_3}
\end{lemma}

\begin{proof}

Для $K = 0$ это верно

\begin{gather*}
\frac{1}{2}\mathbb{E}R_0^2 = \frac{1}{2}R_0^2 \leq \frac{1}{2}P_0^2
\end{gather*}

Воспользуемся следующими 2 фактами, по индукции:
\begin{enumerate}
\item 
\begin{gather*}
	\frac{1}{2}(\mathbb{E}R_k)^2 \leq \frac{1}{2}\mathbb{E}R_k^2 \leq P_0^2;\\
\end{gather*}
\begin{gather}
\label{induction1}
	\mathbb{E}R_k \leq \sqrt{2}P_0
\end{gather}

\item
Так как оптимизация происходит на $\mathds{R}^n$, то $\nabla f(x_*) = 0$, поэтому
\begin{gather*}
M_k = \norm{\nabla f(y_{k+1})} = \norm{\nabla f(y_{k+1}) - \nabla f(x_*) }_2 \leq \sqrt{L}\norm{y_{k+1} - x_*}_L \leq_{{\tiny \circled{1}}} \\ \leq \sqrt{L}\sum_{k=0}^{K - 1} q_k R_k\\ \sum_{k=0}^{K-1} q_k = 1
\end{gather*}
${{\tiny \circled{1}}}$ - следует из Леммы \ref{lemAa3}

В конечном счете
\begin{gather}
\label{induction2}
\mathbb{E}M_k \leq \sqrt{2L} P_0
\end{gather}

\end{enumerate}

Докажем далее по индукции.
Из леммы \ref{lemAZ2_3} возьмем от обеих частей неравенства полное математическое ожидание и просуммируем все неравенства по $k = 0, ..., K - 1$ и воспользуемся (\ref{induction1}), (\ref{induction2}). Зафиксируем $u = x_*$.
\begin{gather*}
		A_{K} \mathbb{E}f(x_K) - A_{0} f(x_0) + \mathbb{E}V(x_*, u_K) - V(x_*, u_0) \leq (A_K - A_0)f_* +\\+ \sum_{k = 0}^{K - 1}\frac{A_{k+1}}{L}\delta^2 +
						\sum_{k = 0}^{K - 1}A_{k+1}\delta \frac{\sqrt{2}P_0}{\sqrt{Ln}} + \sum_{k = 0}^{K - 1}\alpha_{k+1} \frac{\sqrt{2}\sqrt{n}\delta P_0}{\sqrt{L}}
\end{gather*}
Так как $\alpha_k \leq \alpha_K$ и $A_k \leq A_K$ $\forall k \leq K$
\begin{gather*}
		A_{K} \mathbb{E}f(x_K) - A_{0} f(x_0) + \mathbb{E}V(x_*, u_K) - V(x_*, u_0) \leq (A_K - A_0)f_* +\\+ \frac{KA_{K}}{L}\delta^2 +
						KA_{K}\delta \frac{\sqrt{2}P_0}{\sqrt{Ln}} + K\alpha_{K} \frac{\sqrt{2}\sqrt{n}\delta P_0}{\sqrt{L}}
\end{gather*}	
Так как $2A_K = \frac{(K - 1 + 2n)^2 + K - 1}{2n^2} \geq \frac{(K - 1 + 2n)^2}{2n^2} \geq \frac{2n(K - 1 + 2n)}{2n^2} \geq \alpha_K n$.
\begin{gather*}
		A_{K} \mathbb{E}f(x_K) - A_{0} f(x_0) + \mathbb{E}V(x_*, u_K) - V(x_*, u_0) \leq (A_K - A_0)f_* +\\ \frac{KA_{K}}{L}\delta^2 +
						KA_{K}\delta \frac{\sqrt{2}P_0}{\sqrt{Ln}} + KA_{K}\delta \frac{2\sqrt{2} P_0}{\sqrt{Ln}}
\end{gather*}
Из того, что $V(x_*, u_K) \geq 0$ и $\frac{1}{2}P_0^2 = \frac{1}{2}R_0^2 + (1 - \frac{1}{n}) (f(x_0) - f_*) = \frac{1}{2}\norm{u_0 - x_*}_L^2 + (1 - \frac{1}{n}) (f(x_0) - f_*) = V(x_*, u_0) + (1 - \frac{1}{n}) (f(x_0) - f_*)$
\begin{gather*}
	\frac{1}{2}\mathbb{E} R_K^2 \leq \frac{1}{2}P_0^2 + \frac{KA_{K}}{L}\delta^2 +
					KA_{K}\delta \frac{3\sqrt{2}P_0}{\sqrt{Ln}}
\end{gather*}

Используя то, что

\begin{gather*}
K \leq N = \ceil{\frac{\sqrt{2}nP_0}{\sqrt{\epsilon}} + 1 - 2n} \leq \frac{\sqrt{2}nP_0}{\sqrt{\epsilon}}\\
A_K \leq A_N \leq \frac{(N - 1 + 2n)^2}{2n^2} \leq \frac{(\frac{\sqrt{2}nP_0}{\sqrt{\epsilon}} + 1)^2}{2n^2} \leq \frac{4P_0^2}{\epsilon}
\end{gather*}

Два последних слагаемых должны быть меньше или равны $\frac{1}{4}P_0^2$, поэтому
\begin{gather*}
\min\Big\{\frac{P_0\sqrt{L}}{2\sqrt{NA_N}}, \frac{\sqrt{Ln} P_0^2}{12\sqrt{2} A_N N P_0}\Big\} \geq \\\geq
\min\Bigg\{\frac{P_0\sqrt{L}}{2\sqrt{(\frac{\sqrt{2}nP_0}{\sqrt{\epsilon}})\Big(\frac{4P_0^2}{\epsilon}\Big)}},
 \frac{\sqrt{Ln} P_0}{12\sqrt{2}(\frac{\sqrt{2}nP_0}{\sqrt{\epsilon}}) \Big(\frac{4P_0^2}{\epsilon}\Big) }\Bigg\} =\\=
\min\Bigg\{\frac{\epsilon^{\frac{3}{4}}\sqrt{L}}{4\sqrt[4]{2}\sqrt{nP_0}},
 \frac{\epsilon^{\frac{3}{2}}\sqrt{L}}{96\sqrt{n}P_0^2}\Bigg\} = \delta
\end{gather*}

Взяв $\delta$ таким образом, мы в конечном счете получим условие леммы.

\end{proof}

Докажем основную теорему

\begin{proof}
	
	Аналогично, из леммы \ref{lemAZ2_3} возьмем от обеих частей неравенства полное математическое ожидание и просуммируем все неравенства по $k = 0, ..., N - 1$ и воспользуемся (\ref{induction1}), (\ref{induction2}).
	
	\begin{gather*}
		A_{N} \mathbb{E}f(x_N) - A_{0} f(x_0) + \mathbb{E}V(u, u_N) - V(u, u_0) \leq (A_N - A_0)f_* +\\+ \sum_{k = 0}^{K - 1}\frac{A_{k+1}}{L}\delta^2 +
								\sum_{k = 0}^{K - 1}A_{k+1}\delta \frac{\sqrt{2}P_0}{\sqrt{Ln}} + \sum_{k = 0}^{K - 1}\alpha_{k+1} \frac{\sqrt{2}\sqrt{n}\delta P_0}{\sqrt{L}}
	\end{gather*}
	
	Возьмем $u = x_*$. Так как $\mathbb{E}V(u, u_N) \geq 0$, $\alpha_k \leq \alpha_K$, $A_k \leq A_K$ $\forall k \leq K$ и $2A_K \geq \alpha_K n$.
	
	\begin{gather*}
		A_{N} (\mathbb{E}f(x_N) - f_*) \leq P_0^2 + \frac{KA_{K}}{L}\delta^2 +
							KA_{K}\delta \frac{3\sqrt{2}P_0}{\sqrt{Ln}}
	\end{gather*}
	Так как
	\begin{gather*}
	\delta \leq \min\Big\{\frac{P_0\sqrt{L}}{2\sqrt{NA_N}}, \frac{\sqrt{Ln} P_0^2}{12\sqrt{2} A_N N P_0}\Big\}
	\end{gather*}
	
	\begin{gather*}
		\mathbb{E}f(x_N) - f_* \leq \frac{P_0^2}{A_{N}} +\frac{1}{4}\frac{P_0^2}{A_{N}} + \frac{1}{4}\frac{P_0^2}{A_{N}}
	\end{gather*}
	
	$N$ выбиралось таким образом, чтобы $\frac{\frac{1}{2}P_0^2}{A_{N}} \leq \epsilon$, поэтому 
	
	\begin{gather*}
			\mathbb{E}f(x_N) - f_* \leq 3\epsilon
		\end{gather*}
	
\end{proof}

\begin{corollary}
\leavevmode

Отметим два практически важных случая:

\begin{enumerate}
\item Предположим, что $e_{k+1}$ распределены равномерно на ортах, то есть равновероятно разыгрывается случайная координата $i \in [1 \dots n]$, тогда на каждой итерации $e_{k+1}$ выбирается следующим образом
\begin{gather*}
e_{k+1}^j =\begin{cases} 
1, & j = i \\
0,& j \neq i \end{cases}
\end{gather*}

Данный метод соответствует координатному методу, в самом деле (пусть $\delta = 0$):
\begin{gather*}
\widetilde{\nabla}f(y) = n(\langle\nabla f(y), e_{k+1}\rangle) e_{k+1} = n \sum_{j = 1}^{n}\frac{\partial f(y)}{\partial y_j} e_{k+1}^j e_{k+1} = n \frac{\partial f(y)}{\partial y_i} e_{k+1}
\end{gather*}

То есть
\begin{gather*}
\{\widetilde{\nabla}f(y)\}_j = \begin{cases} 
n\frac{\partial f(y)}{\partial y_i}, & j = i \\
0,& j \neq i \end{cases}
\end{gather*}

\item
Рассмотрим другой, не менее важный пример, связанный с безградиентным методом. Будем предполагать, что у нас нет доступа к градиенту функции, поэтому мы попробуем оценить истинный градиент с помощью разностной аппроксимации следующим образом
\begin{gather*}
\widetilde{\nabla} f(x) = \frac{n}{\tau}((f(x+\tau e_{k+1}) + \delta^1_{k+1}) - (f(x) + \delta^2_{k+1})) e_{k+1},
\end{gather*}
где $e_{k+1}$ - случайный вектор равномерно распределенный на сфере.
Если бы $\delta^1_{k+1} = \delta^2_{k+1} = 0$, то мы смогли просто устремить $\tau$ к нулю. Но на практике все вычисления функций происходит с некоторой точностью, поэтому и возникают ненулевые слагаемые $\delta^1_{k+1}$ и $\delta^2_{k+1}$, о которых только известно, что они ограничены $\delta$ (например, это может быть машинной точностью ЭВМ).

Приведем данную аппроксимацию градиента к стандартному виду, чтобы могли применить Теорему \ref{mainTheorem3}.
\begin{gather*}
\widetilde{\nabla} f(x) = \frac{n}{\tau}((f(x+\tau e_{k+1}) + \delta^1_{k+1}) - (f(x) + \delta^2_{k+1})) e_{k+1} = \\=
n \langle \nabla f(x), e_{k+1}\rangle e_{k+1} + \frac{n}{\tau}(f(x+\tau e_{k+1}) - f(x) - \tau \langle \nabla f(x), e_{k+1}\rangle +\\
+ \delta^1_{k+1} - \delta^2_{k+1})e_{k+1}
\end{gather*}

Возьмем $\widetilde{\delta}_{k+1} = \frac{1}{\tau}(f(x+\tau e_{k+1}) - f(x) - \tau \langle \nabla f(x), e_{k+1}\rangle + \delta^1_{k+1} - \delta^2_{k+1})$. Оценим $\abs{\widetilde{\delta}_{k+1}}$:
\begin{gather*}
\abs{\widetilde{\delta}_{k+1}} = \abs{\frac{1}{\tau}(f(x+\tau e_{k+1}) - f(x) - \tau \langle \nabla f(x), e_{k+1}\rangle + \delta^1_{k+1} - \delta^2_{k+1})} \leq \\ \leq
\frac{1}{\tau}\abs{f(x+\tau e_{k+1}) - f(x) - \tau \langle \nabla f(x), e_{k+1}\rangle} + \frac{1}{\tau}\abs{\delta^1_{k+1} - \delta^2_{k+1}}
\end{gather*}
Из Липшивости и выпуклости $0 \leq f(x+\tau e_{k+1}) - f(x) - \tau \langle \nabla f(x), e_{k+1}\rangle \leq \frac{L \tau^2}{2}$, поэтому
\begin{gather*}
\abs{\widetilde{\delta}_{k+1}} \leq \frac{L \tau^2}{2 \tau} + \frac{2 \delta}{\tau} =
\frac{L \tau}{2} + \frac{2 \delta}{\tau}
\end{gather*}
Обозначим $\hat{\delta} = \frac{L \tau}{2} + \frac{2 \delta}{\tau}$. Чтобы выполнялась Теорема \ref{mainTheorem3}, нужно, чтобы выполнялось следующее условие
\begin{gather*}
\hat{\delta} \leq \min\Bigg\{\frac{\epsilon^{\frac{3}{4}}\sqrt{L}}{4\sqrt[4]{2}\sqrt{nP_0}},
	 \frac{\epsilon^{\frac{3}{2}}\sqrt{L}}{96\sqrt{n}P_0^2}\Bigg\}
\end{gather*}

Проминимизируем $\hat{\delta}$ по $\tau$, тогда $\hat{\delta} = 2 \sqrt{L \delta}$ и $\tau = 2\sqrt{\frac{\delta}{L}}$. То есть достаточно взять $\tau = 2\sqrt{\frac{\delta}{L}}$ и $\delta$:

\begin{gather*}
\hat{\delta} = 2 \sqrt{L \delta} \leq \min\Bigg\{\frac{\epsilon^{\frac{3}{4}}\sqrt{L}}{4\sqrt[4]{2}\sqrt{nP_0}},
	 \frac{\epsilon^{\frac{3}{2}}\sqrt{L}}{96\sqrt{n}P_0^2}\Bigg\}\\
\delta \leq \min\Bigg\{\frac{\epsilon^{\frac{3}{2}}}{64\sqrt{2} n P_0},
	 \frac{\epsilon^{3}}{36864 n P_0^4}\Bigg\}
\end{gather*}

То есть, чтобы обеспечить сходимость метода, ошибка $\delta$ должна быть порядка $\mathcal{O}(\frac{\epsilon^{3}}{n})$, а $\tau$ - $\mathcal{O}(\frac{\epsilon^{\frac{3}{2}}}{\sqrt{n}})$.

\end{enumerate}

\end{corollary}

\section{Заключение}

В данной работе были представлены модификации зеркального метода треугольника из Раздела \ref{sec:mmt}. Стоит отметить, что базовый метод из Раздела \ref{sec:mmt} был независимо предложен ранее в \cite{lan2012optimal}, но незначительные изменения базового алгоритма, проделанные нами, позволяют получать оценки быстрого градиентного метода для различных задач, в частности, нам удалось получить алгоритм адаптивного быстрого градиентного метода для задачи минимакса, далее довольно просто обобщить метод для $(\delta, L)$-оракула. К этому всему нам удалось представить метод для случая, когда вместо градиента у нас имеется некоторая случайная оценка с шумом, в конечном счете удалось получить ограничения на шум, при которых метод имеет скорость сходимости быстрого градиентного метода.

\bibliography{bibl}{}
\bibliographystyle{plain}

\appendix
\section{Доказательство для зеркального метода треугольника} \label{appendix:profMMT}
		
\begin{lemma}
	\label{lemAa2}
	Пусть для последовательности $\alpha_k$ выполнено
	\begin{align*}
	A_k = \sum_{i = 0}^{k}\alpha_i\\
	\alpha_0 = 0\\
	A_k = L\alpha_k^2
	\end{align*}
	Тогда верно следующее рекуррентное соотношение $\forall k \geq 0$

	 \begin{align*}
	 \alpha_{k+1} = \frac{1}{2L} + \sqrt{\frac{1}{4L^2} + \alpha_k^2}
	 \end{align*}
	 и $\forall k \geq 1$
	 \begin{gather*}
	 \alpha_{k} \geq \frac{k+1}{2L}\\
	 A_k \geq \frac{(k+1)^2}{4L}
	 \end{gather*}
\end{lemma}

\begin{proof}
	\begin{equation*}
		L\alpha^2_{k+1} = A_{k+1}
	\end{equation*}
	\begin{equation*}
		L\alpha^2_{k+1} = A_{k} + \alpha_{k}
	\end{equation*}
	\begin{equation*}
		L\alpha^2_{k+1} - \alpha_{k+1} - A_{k} = 0 
	\end{equation*}
	Решая данное квадратное уравнение получаем, что
	\begin{equation*} 
		\alpha_{k+1} = \frac{1 \pm \sqrt{\uprule 1 + 4LA_{k}}}{2L}
	\end{equation*}
	\begin{equation*}
		\alpha_{k+1} = \frac{1}{2L} + \sqrt{\frac{1}{4L^2} + \frac{A_{k}}{L}} = \frac{1}{2L} + \sqrt{\frac{1}{4L^2} + \alpha_k^2}
	\end{equation*}
	Если $k = 0$, то получаем, что $\alpha_{1} = \frac{1}{L}$ и $A_1 \geq \frac{1}{L}$, база индукции верна. Пусть данная лемма верна для $k$, докажем для $k+1$:
	\begin{equation*}
	\alpha_{k+1} = \frac{1}{2L} + \sqrt{\frac{1}{4L^2} + \alpha_k^2} \geq \frac{1}{2L} + \alpha_k
	\end{equation*}
	Из того, что
	\begin{equation*}
	\alpha_{k} \geq \frac{k+1}{2L},
	\end{equation*}
	получаем:
	\begin{equation*}
	\alpha_{k+1} \geq \frac{k + 2}{2L}
	\end{equation*}
	и
	\begin{equation*}
	A_{k+1} = L \alpha_{k+1}^2 \geq \frac{(k + 2)^2}{4L}
	\end{equation*}
\end{proof}

\begin{lemma}
	$\forall u \in Q$ выполнено
	\begin{equation*}
	\alpha_{k+1}\langle \nabla f(y_{k+1}), u_k - u\rangle \leq A_{k+1}(f(y_{k+1}) - f(x_{k+1})) + V(u, u_k) - V(u, u_{k+1}) 
	\end{equation*}
	\label{lemAZ1}
\end{lemma}

\begin{proof}
	\begin{gather*}
	\alpha_{k+1}\langle \nabla f(y_{k+1}), u_k - u\rangle = \\=
	\alpha_{k+1}\langle \nabla f(y_{k+1}), u_k - u_{k+1}\rangle + \alpha_{k+1}\langle \nabla f(y_{k+1}), u_{k+1} - u\rangle \leq_{{\tiny \circled{1}}} \\\leq
	\alpha_{k+1}\langle \nabla f(y_{k+1}), u_k - u_{k+1}\rangle + 
	\langle -\nabla_{u_{k+1}} V(u_{k+1}, u_k), u_{k+1} - u \rangle = \\=
	 \alpha_{k+1}\langle \nabla f(y_{k+1}), u_k - u_{k+1}\rangle + V(u, u_k) - V(u, u_{k+1}) - V(u_{k+1}, u_k) \leq \\\leq
	 \alpha_{k+1}\langle \nabla f(y_{k+1}), u_k - u_{k+1}\rangle + V(u, u_k) - V(u, u_{k+1}) - \frac{1}{2}\norm{u_k - u_{k+1}}^2 =_{{\tiny \circled{2}}} \\=
	 A_{k+1}\langle \nabla f(y_{k+1}), y_{k+1} - x_{k+1}\rangle + V(u, u_k) - V(u, u_{k+1}) - \frac{A_{k+1}^2}{2\alpha_{k+1}^2}\norm{y_{k+1} - x_{k+1}}^2 = \\=
	 A_{k+1}(\langle \nabla f(y_{k+1}), y_{k+1} - x_{k+1}\rangle - \frac{L}{2}\norm{y_{k+1} - x_{k+1}}^2) + V(u, u_k) - V(u, u_{k+1}) \leq_{{\tiny \circled{3}}} \\\leq
	 A_{k+1}(f(y_{k+1}) - f(x_{k+1})) + V(u, u_k) - V(u, u_{k+1})
	\end{gather*}
\end{proof}

{\small \circled{1}} - из условия оптимальности (\ref{equmir})

{\small \circled{2}} - из (\ref{eqxmir}) и (\ref{eqymir})

{\small \circled{3}} - условие Липшица

\begin{lemma}
	$\forall u \in Q$ выполнено
	\begin{equation*}
		A_{k+1} f(x_{k+1}) - A_{k} f(x_{k}) + V(u, u_{k+1}) - V(u, u_{k}) \leq \alpha_{k+1}f(u)
	\end{equation*}
	\label{lemAZ2}
\end{lemma}

\begin{proof}
	\begin{gather*}
		\alpha_{k+1}(f(y_{k+1}) - f(u)) \leq \\\leq
		\alpha_{k+1}\langle \nabla f(y_{k+1}), y_{k+1} - u\rangle = \\=
		\alpha_{k+1}\langle \nabla f(y_{k+1}), y_{k+1} - u_k\rangle + 
		\alpha_{k+1}\langle \nabla f(y_{k+1}), u_k - u\rangle =_{{\tiny \circled{1}}} \\=
		A_{k}\langle \nabla f(y_{k+1}), x_{k} - y_{k+1}\rangle + 
		\alpha_{k+1}\langle \nabla f(y_{k+1}), u_k - u\rangle \leq \\\leq
		A_{k}(f(x_k) - f(y_{k+1})) + \alpha_{k+1}\langle \nabla f(y_{k+1}), u_k - u\rangle \leq_{{\tiny \circled{2}}} \\\leq
		A_{k}(f(x_k) - f(y_{k+1})) + A_{k+1}(f(y_{k+1}) - f(x_{k+1})) + V(u, u_k) - V(u, u_{k+1})  = \\=
		\alpha_{k+1}f(y_{k+1}) + A_k f(x_{k}) - A_{k+1} f(x_{k+1}) + V(u, u_k) - V(u, u_{k+1})
	\end{gather*}
\end{proof}

{\small \circled{1}} - из (\ref{eqymir})

{\small \circled{2}} - из Леммы \ref{lemAZ1}

\begin{theorem}
	\begin{equation*}
	f(x_N) - f(x_*) \leq \frac{4LR^2}{(N+1)^2}
	\end{equation*}
\end{theorem}
\begin{proof}
	
	Просуммируем нер-во из леммы \ref{lemAZ2} по $k = 0, ..., N - 1$
	
	\begin{gather*}
		A_{N} f(x_N) - A_{0} f(x_0) + V(u, u_N) - V(u, u_0) \leq (A_N - A_0)f(u) \\
		A_{N} f(x_N) + V(u, u_N) - V(u, u_0) \leq A_Nf(u)
	\end{gather*}
	
	Возьмем $u = x_*$, воспользуемся тем, что $V(x_*, u_N) \geq 0$ и $u_0 = x_0$, тогда
	
	\begin{gather*}
		A_{N} (f(x_N) - f_*) \leq R^2\,\,\,\{\stackrel{def}{=} V(x_*, x_0)\}
	\end{gather*}
	
\end{proof}

\begin{corollary}

Рассмотрим неравенство
\begin{gather*}
A_{N} f(x_N) + V(u, u_N) - V(u, u_0) \leq A_Nf(u)
\end{gather*}

Возьмем $u = x_*$, воспользуемся тем, что $f(x_N) \geq f(x_*)$, тогда

\begin{gather*}
V(x_*, u_{N}) \leq V(x_*, u_{0})
\end{gather*}

То есть последовательность $u_N$ ограничена.

\begin{gather*}
\frac{1}{2}\norm{x_* - u_{N}}^2 \leq V(x_*, u_{N}) \leq R^2
\end{gather*}

По индукции:

\begin{gather*}
\frac{1}{2}\norm{x_* - x_{k+1}}^2 = \frac{1}{2}\norm{x_* - \frac{\alpha_{k+1}u_{k+1} + A_k x_k}{A_{k+1}}}^2 =\\= \frac{1}{2}\norm{\frac{\alpha_{k+1}(x_* - u_{k+1}) + A_k(x_* - x_k)}{A_{k+1}}}^2 \leq \\\leq
\frac{1}{2}\frac{\alpha_{k+1}}{A_{k+1}}\norm{x_* - u_{k+1}}^2 + \frac{1}{2}\frac{A_k}{A_{k+1}}\norm{x_* - x_k}^2 \leq R^2
\end{gather*}

То есть последовательность, генерированная методом, ограничена.

\end{corollary}

\end{document}